\documentclass[a4paper,11pt,leqno]{article}
\usepackage{amssymb, amsmath, amsthm, graphicx}
\usepackage{color, epsfig, amsfonts}
\usepackage{enumitem}
\usepackage{hyperref}

\newtheorem{thm}{Theorem}[section]
\newtheorem{lem}[thm]{Lemma}
\newtheorem{cor}[thm]{Corollary}

\newtheorem{exa}{Example}[section]
\numberwithin{equation}{section}

\renewcommand{\a}{\alpha}
\renewcommand{\b}{\beta}
\newcommand{\e}{\varepsilon}
\newcommand{\de}{\delta}

\newcommand{\si}{\sigma}
\newcommand{\Si}{\Sigma}
\renewcommand{\t}{\tau}

\newcommand{\De}{\Delta}
\newcommand{\Ga}{\Gamma}
\newcommand{\La}{\Lambda}

\def\R{{\mathbb{R}}}
\def\N{{\mathbb{N}}}
\def\Z{{\mathbb{Z}}}
\def\T{{\mathbb{T}}}

\def\E{{\mathbb E}}

\allowdisplaybreaks

\newcommand{\supp}{\operatorname{supp}}

\newcommand{\gap}{\operatorname{gap}}

\definecolor{darkgreen}{rgb}{0,0.7,0}
\definecolor{orange}{rgb}{1,0.45,0}

\title{Motion by mean curvature from \\
Glauber-Kawasaki dynamics with speed change}
\author{Tadahisa Funaki$\,^{1)}$, Patrick van Meurs$\,^{2)}$, Sunder Sethuraman$\,^{3)}$ and Kenkichi Tsunoda$\,^{4)}$}

\begin{document}
\maketitle




\begin{abstract}
We derive a continuum mean-curvature flow as a certain hydrodynamic scaling limit of Glauber-Kawasaki dynamics with speed change. The Kawasaki part describes the movement of particles through particle interactions. It is speeded up in a diffusive space-time scaling. The Glauber part governs the creation and annihilation of particles. The Glauber part is set to favor two levels of particle density. It is also speeded up in time, but at a lesser rate than the Kawasaki part. Under this scaling, a mean-curvature interface flow emerges, with a homogenized `surface tension-mobility' parameter reflecting microscopic rates. The interface separates the two levels of particle density. 

Similar hydrodynamic limits have been derived in two recent papers; one where the Kawasaki part describes simple nearest neighbor interactions, and one where the Kawasaki part is replaced by a zero-range process. We extend the main results of these two papers beyond nearest-neighbor interactions. The main novelty of our proof is the derivation of a `Boltzmann-Gibbs' principle which covers a class of local particle interactions.

\footnote{
\hskip -6mm 
${}^{1)}$ Department of Mathematics, Waseda University,
3-4-1 Okubo, Shinjuku-ku, Tokyo 169-8555, Japan. \\
New address: Beijing Institute of Mathematical Sciences and Applications, 
Huairou district, Beijing, China. \\
e-mail: funaki@ms.u-tokyo.ac.jp \\
${}^{2)}$ Faculty of Mathematics and Physics, Kanazawa University, Kakuma, Kanazawa 920-1192, Japan. \\
e-mail: pjpvmeurs@staff.kanazawa-u.ac.jp \\
${}^{3)}$ Department of Mathematics, University of Arizona,
621 N.\ Santa Rita Ave., Tucson, AZ 85750, USA. \\
e-mail: sethuram@math.arizona.edu\\
${}^{4)}$ Faculty of Mathematics, Kyushu University, 744, Motooka, Nishi-ku, Fukuoka
819-0395, Japan. \\
e-mail: tsunoda@math.kyushu-u.ac.jp}
\footnote{
\hskip -6mm
Keywords: Hydrodynamic limit, Motion by mean curvature, Glauber-Kawasaki dynamics, Sharp interface limit.}
\footnote{
\hskip -6mm
Abbreviated title $($running head$)$: MMC from Glauber-Kawasaki dynamics}
\footnote{
\hskip -6mm
2020MSC: 60K35, 82C22, 74A50.}
\end{abstract}




\paragraph{Acknowledgements}

TF was supported in part by JSPS KAKENHI Grant Number JP18H03672.
PvM was supported by JSPS KAKENHI Grant Number JP20K14358. 
SS was supported by grant ARO W911NF-181-0311.
KT was supported by JSPS KAKENHI Grant Number JP18K13426 and JP22K13929.

\paragraph{Data availability statement} All data generated or analysed during this study are included in this published article 

\section{Introduction}
\label{s:intro}

We are interested in the hydrodynamic limit for Glauber-Kawasaki dynamics with speed change
under a certain scaling which leads to the motion by mean curvature for
a phase-separating interface appearing in the limit.
Glauber-Kawasaki dynamics describes a Markovian particle system on the microscopic scale where the particles move on the $d$-dimensional
discrete torus $\T_N^d :=(\Z/N\Z)^d = \{1,2,\ldots,N\}^d$ of size $N$. The configuration of the particles is described by $\eta  = \{ \eta_x \in \{0,1\} \}_{x \in \T_N^d}$, where $\eta_x = 0$ indicates that the site $x$ is vacant, and $\eta_x = 1$ indicates that the site $x$ is occupied by a particle. The generator of this process is expressed as 
\begin{equation} \label{LN}
  L_N = N^2 L_K + K L_G.
\end{equation}
The Kawasaki part, $L_K$, describes the rate $c_{x,y}(\eta)$ at which the particle occupations at neighboring sites $x$ and $y$ are swapped. If one site contains a particle and the other does not, then this swapping corresponds to a particle hopping from one site to the other. The Glauber part, $L_G$, describes the flip rate $c_x(\eta)$ at which $\eta_x$ flips from $0$ to $1$ and vice versa. This corresponds to birth and death of a particle at site $x$. 
In \eqref{eq:1.4} we give the precise definitions of $L_K$ and $L_G$.
The hydrodynamic limit corresponds to the limit $N \to \infty$ of these dynamics. The prefactor $N^2$ in front of $L_K$ is the natural scaling for which the Kawasaki part results in diffusion of the particle density on the macroscopic scale. We are interested in the regime where 
$1\le K=K(N) \to \infty$ as $N \to \infty$. 

Before describing the case $K \to \infty$ as $N \to \infty$, it is instructive to consider first the case of constant $K$ (i.e., independent of $N$). Under certain assumptions on the exchange rates $c_{x,y}(\eta)$ and the flip rates $c_x(\eta)$, one may expect from \cite{DFL,FHU} that the hydrodynamic limit is given by
\begin{equation}  \label{eq:RD-b}
\partial_t\rho= \De P(\rho) + Kf(\rho), \quad v\in \T^d.
\end{equation}
Here, $\rho = \rho(t,v) \in [0,1]$
is the particle density on the macroscopic domain $\T^d \: ( =[0,1)^d$ with periodic boundary), and the Laplace operator $\Delta$ acts on the variable $v$. Microscopic sites $x \in \T_N^d$ correspond on the macroscopic scale to the $d$-dimensional cube of size $1/N$ centered at $v = x/N \in \T^d$. The functions $P$ and $f$ can be explicitly expressed in terms of $c_{x,y}$ and $c_x$; see \eqref{eq:P} and \eqref{eq:f}.

We are interested in those $P$ and $f$ for which the reaction-diffusion equation \eqref{eq:RD-b} is a nonlinear Allen-Cahn equation. The width of the interface is proportional to $1/\sqrt K$. 
Then, as $K \to \infty$, it is shown in \cite{EFHPS, EFHPS-2} that, under certain assumptions on $P$ and $f$ 
(in particular the bistability and the balance condition; see (BS) below),
the solution of \eqref{eq:RD-b} converges to a sharp-interface limit; see \eqref{MCF}. 

The aim of this paper is to establish the hydrodynamic limit of the process generated by $L_N$ as $N \to \infty$ in the case where $K = K(N) \to \infty$ as $N \to \infty$; see Theorem \ref{thm:main}. We expect the limit to be the same as the sharp-interface limit of \eqref{eq:RD-b}. This limit is the mean curvature flow of a closed hypersurface $\Gamma_t \subset \T^d$ given by
\begin{equation} \label{MCF}  
  V = \lambda_0 \kappa \quad \text{on } \Gamma_t,
\end{equation}
where $V$ is the normal velocity of $\Gamma_t$, $\kappa$ is the mean curvature of $\Gamma_t$ multiplied by $d-1$, and $\lambda_0$ is defined by
\begin{equation*} 
  \lambda_0
  := \frac{ \int_{\alpha_1}^{\alpha_2} P'(\alpha) \sqrt{W(\alpha)} \, d \alpha }{ \int_{\alpha_1}^{\alpha_2} \sqrt{W(\alpha)} \, d \alpha },
  \qquad W(\alpha) := \int_{\alpha}^{\alpha_2} f(\b) P'(\b) \, d \b.
\end{equation*}
As explained in the appendix of \cite{EFHPS-2}, $\lambda_0$ is the product of the surface tension (proportional to the numerator) multiplied by the mobility (proportional to the reciprocal of the denominator).
Under our assumptions it turns out that $\lambda_0$ is well-defined and positive. Since $P$ and $f$ are determined in terms of the rates $c_{x,y}$ and $c_x$, our desired hydrodynamic limit will give a precise relation between the microscopic rates $c_{x,y}$ and $c_x$ and the resulting mobility of the interface on the macroscopic scale.
 
Our aim builds further on the literature on understanding curvature flows (or PDEs in general) as hydrodynamic limits of particle systems; see e.g. \cite{KS,Bon,FT,DFPV,Hay,EFHPS}. In particular, two of these papers are close to our aim. 

The first is \cite{FT}, which considers a less general setting in which $c_{x,y}(\eta) \equiv 1$. In that setting, the diffusion in the corresponding equation \eqref{eq:RD-b} is linear, i.e.\ $P(\rho) = \rho$. Our aim can therefore be reformulated as the generalization of the result in \cite{FT} to exchange rates $c_{x,y}(\eta)$ which may depend on the site $x$, the neighboring site $y$ and a finite-range dependence on the microscopic particle configuration.
The dependence of $c_{x,y}(\eta)$ on $\eta$ is called \lq speed change' in Kawasaki dynamics.
We will also allow for more general flip rates $c_x$. The main reason for pursuing this generalization is to discover the precise dependence of $\lambda_0$ on $c_{x,y}$ and $c_x$.

The second paper from the list is \cite{EFHPS}, where the only difference with our setting is that the Kawasaki part is replaced by a zero-range process where $c_{x,y}(\eta)$ depends, in terms of $\eta$, only on the single site variable $\eta_x$. While a zero-range process also gives rise to nonlinear diffusion on the macroscopic scale, the finite-range feature of the Kawasaki process is more natural than the corresponding zero-range feature in \cite{EFHPS}. Establishing the hydrodynamic limit with respect to Kawasaki interactions will therefore show a certain robustness of the mean-curvature scaling limit.  Moreover, the intrinsic difference with the zero-range process poses challenges which we overcome in various ways by combining and extending the techniques in \cite{EFHPS,FT}.

\subsection{The Glauber-Kawasaki dynamics}
\label{s:intro:GK}

Here we describe the Glauber-Kawasaki dynamics with speed change in full detail.
The configuration space is $\mathcal{X}_N = \{0,1\}^{\T_N^d}$.
We denote its elements by $\eta=\{\eta_x\}_{x\in\T_N^d}$.
The generator is given by $L_N$ in \eqref{LN}, where
\begin{equation}  \label{eq:1.4}
\begin{aligned}
L_Kf (\eta)   
&  = \frac12 \sum_{ \substack{ x, y\in\T_N^d \\ |x-y|=1 }} c_{x,y}(\eta)
\left\{  f\left(  \eta^{x,y}\right)  -f\left(\eta\right)  \right\},   \\
L_Gf (\eta)  
&  =\sum_{x\in\T_N^d} c_x(\eta)
  \left\{  f\left(  \eta^x \right)  -f\left(  \eta\right)  \right\}
\end{aligned}
\end{equation}
for functions $f$ on $\mathcal{X}_N$ (not to be confused with $f$ of the reaction term in \eqref{eq:RD-b}). 

Next we explain the symbols $\eta^{x,y}$, $\eta^{x}$, $c_{x,y}$ and $c_x$. First, $\eta^{x,y} \in \mathcal{X}_N$ is the 
configuration $\eta$ after an exchange happens 
between $x$ and $y$, i.e., 
\[
  (\eta^{x,y})_z 
  = \left\{ \begin{array}{ll}
    \eta_y
    &\text{if } z = x \\
    \eta_x
    &\text{if } z = y \\
    \eta_z
    &\text{otherwise.}
  \end{array} \right.
\]
Second, $\eta^x  \in \mathcal{X}_N$ is the configuration after a flip happens at $x$, i.e., 
\[
  (\eta^x)_z
  = \left\{ \begin{array}{ll}
    1-\eta_x
    &\text{if } z = x \\
    \eta_z
    &\text{otherwise.}
  \end{array} \right.
\]
Third, $c_{x,y}(\eta)$ are the bond-dependent (i.e.\ $c_{x,y}=c_{y,x}$) exchange rates of configurations at $x$ and $y$.
The choice of $y \in \T_N^d$ is limited to $|x - y|=1$, which means that $x$ and $y$ are neighboring sites. We regard $c_{x,y}$ as functions on the infinite configuration space 
$\mathcal{X}=\{0,1\}^{\Z^d}$. To list our assumptions on $c_{x,y}$, we set $e_i\in \Z^d$ as the unit vectors  of the direction $i$, and introduce $\t_x: \mathcal{X} \to \mathcal{X}$ (or $\mathcal{X}_N \to \mathcal{X}_N$)
as the translation of $\eta$ by $x\in \Z^d$ (or $\in \T_N^d$) defined by
$(\t_x\eta)_z= \eta_{z+x}$ for $z\in \Z^d$ (or $\in \T_N^d$, in which case $x+z$ is taken modulo $N$). For a function $f=f(\eta)$ on $\mathcal{X}$ or $\mathcal{X}_N$, 
we denote $\t_xf(\eta) := f(\t_x\eta)$. Our assumptions on $c_{x,y}$ are as follows. For all $\eta \in \mathcal X$ and all $x,y \in \T_N^d$ with $|x - y|=1$, we assume:
\begin{itemize}
\item[(1)] (non-degeneracy) $c_{x,y}(\eta)>0$.
\item[(2)] (translation-invariance) $c_{x,y}(\eta) = c_{0,y-x}(\t_x\eta)$.
\item[(3)] (finite range) $c_{0,y-x}$ is a local function on $\mathcal{X}$, i.e.\ $c_{0,y-x}$ has finite support. More precisely,
$c_{0,y-x}(\eta) = c_{0,y-x}(\eta_\La)$ for some universal $\La \subset \Z^d$ with finite cardinality, where $\eta_\Lambda := \{ \eta_x \}_{x \in \Lambda}$.
\item[(4)] (reversibility under Bernoulli measures)  $c_{x,y}(\eta)$ does not depend on 
$(\eta_x,\eta_y)$, i.e.\ $c_{x,y}(\eta^{x,y}) = c_{x,y}(\eta)$.
\item[(5)] (gradient condition) There exist local functions
$\{h_i(\eta)\}_{i=1}^d$ on $\mathcal{X}$ such that the current at the $e_i$-directed bond (edge) $(0,e_i)$
defined by $c_{0,e_i}(\eta) (\eta_0-\eta_{e_i})$ (cf.\ \cite[(GS) and Remark 5.3]{FHU})
can be written as $h_i(\eta)- h_i(\t_{e_i} \eta)$ for each $i$.
\end{itemize}
We wish to make four remarks on these five assumptions. First, due to $c_{x,y} = c_{y,x}$ and (2), all the rates $c_{x,y}$ are characterized by $c_{0,e_i}$ for $i = 1,\ldots,d$. Second, (1), (2) and (3) imply 
\begin{equation} \label{cxy:LB}
  c_{x,y}(\eta) \geq c_*
\end{equation}
for some positive constant $c_* > 0$ independent of $x,y,\eta$. Third, we will always assume that $N$ is large enough such that the ranges of the local functions in (3) and (5) fit inside (a translated copy of) $\T_N^d$. Fourth, (4) can alternatively be understood as a condition under which the 
Bernoulli measures $\nu_\rho$ satisfy detailed balance and are reversible. 

Finally, the fourth symbol $c_x(\eta)$ appears in $L_G$ and represents the flip rate. We assume that $c_x(\eta) = \t_x c(\eta)$ for some non-negative 
local function $c = c_0$ on $\mathcal{X}$ (regarded as that on
$\mathcal{X}_N$ for $N$ large enough). Since $\eta_0$ takes values in $\{0,1\}$, $c(\eta)$ can be decomposed as
\begin{equation}  \label{eq:c+-}
c(\eta) = c^+(\eta)(1-\eta_0)+c^-(\eta) \eta_0
\end{equation}
for some local functions $c^\pm$ which do not depend on $\eta_0$.
We interpret $c^+(\eta)$ and $c^-(\eta)$ as the rates of creation and annihilation
of a particle at $x=0$, respectively. This completes the description of the expression of the generator $L_N$ in \eqref{LN}. We remark that $N$ needs to be large enough so that the supports of $c_0$ and $c_{0,e_i}$ for $i=1,\ldots,d$ are contained in $\{0,1\}^\La$ for some box $\La$ centered at $0$ with side length $\ell < N$. However, this issue is of little importance since our aim is to study the limit $N \to \infty$. 
\smallskip 

Our list of assumptions on $c_{x,y}$ and $c$ is not complete yet. As anticipated in the introduction, we require several properties of the functions $P$ and $f$ in \eqref{eq:RD-b}. These properties will result into additional assumptions on $c_{x,y}$ and $c$. 

By extrapolating the results in \cite{DFL,FHU} 
 for a fixed $K\ge 1$ independent of $N$, one may expect that the empirical density of the process $\eta^N(t)$ generated by $L_N$ converges to a solution
of the generalized version of \eqref{eq:RD-b} given by 
\begin{equation}  \label{eq:RD}
\partial_t\rho= \sum_{i=1}^d \frac{\partial^2}{\partial v_i^2} P_i(\rho) + K f(\rho), \quad v\in \T^d,
\end{equation}
where the functions $P_i$ and $f$ are defined by
\begin{equation}  \label{eq:P}
P_i(\rho) := E^{\nu_\rho}[h_i(\eta)], 
\end{equation} 
\begin{align} \label{eq:f}
  f(\rho) 
&:= E^{\nu_\rho}[(1-2\eta_0) c(\eta)]    \\\notag
&=  E^{\nu_\rho}[c^+(\eta) (1-\eta_0) -c^-(\eta) \eta_0] \\ \notag
& = (1-\rho) E^{\nu_\rho}[c^+(\eta)] - \rho E^{\nu_\rho}[c^-(\eta)]
\end{align}
respectively, where $\nu_\rho$ is the Bernoulli measure on $\mathcal{X}$ 
with mean $\rho\in [0,1]$, and, for any distribution $\nu$, $E^\nu$ is the expectation with respect to $\nu$. Note that the functions $P_i$ and $f$ are completely characterized in terms of $c_{x,y}(\eta)$ and $c(\eta)$ respectively. Since $h_i$ and $c$ are local functions, $P_i$ and $f$ are polynomials. We demonstrate in Examples \ref{Ex:2.1} and \ref{exa:4.1} that our assumptions on $c_{x,y}$ and $c$ cover various choices for $P_i$ and $f$. 

\begin{exa}  \label{Ex:2.1}
We provide two examples for the exchange rates $c_{x,y}(\eta)$ satisfying assumptions {\rm (1)--(5)}, and compute the resulting functions $P_i$:
\begin{enumerate}
  \item $c_{x,y}(\eta) \equiv 1$. This corresponds to the original Kawasaki dynamics. For the functions $h_i$ in assumption {\rm (5)} one can simply take
$h_i(\eta) =\eta_0$ for all $i$, which by \eqref{eq:P} results in $P_i(\rho) = \rho$.
  \item $c_{0,e_i}(\eta) = 1+ \a_i(\eta_{-e_i}+\eta_{2e_i})$ for $ 
\a_i> - 1/2$; see \cite[Remark 5.3]{FHU}. A possible choice for the functions $h_i$ is
$$
h_i(\eta) = (\a_i+1)\eta_0+\a_i (\eta_{-e_i}-\eta_0)(\eta_0-\eta_{e_i}),
$$
which by \eqref{eq:P} results in $P_i(\rho)=\rho+\a_i \rho^2$.
\end{enumerate}
\end{exa} 

The assumptions on $P_i$ (and therefore on $c_{x,y}$) are:
\begin{itemize}
\item[(6)] $P_i=P$ for $i=1,\ldots,d$, i.e.\ $P_i$ is independent of $i$.
\end{itemize}
Note that both examples in Example \ref{Ex:2.1} can be made to fit to (6). 

Assumption (6) has three convenient consequences. First, the PDE \eqref{eq:RD} reduces to the form \eqref{eq:RD-b}.
Note that the diffusion matrix $D(\rho)=\{D_{ij}(\rho)\}_{i,j=1}^d$ of \eqref{eq:RD-b} is a diagonal matrix with diagonal elements all equal to $P'(\rho)$, i.e.\
$$
D_{ij}(\rho) = P'(\rho) \de_{ij}.
$$
Second, the $P$-balance condition for the reaction term 
$f(\rho)$ (see the condition (BS) below) becomes independent of $i$. Third, we claim that $P'(\rho) > 0$. To prove this claim, we first recall the Green-Kubo formula.

\begin{lem}[Green-Kubo formula {\cite[(1.12)]{FHU}}] \label{lem:GK}
For $i=1,\ldots,d$
$$
P_i'(\rho) =
\frac1{2\rho(1-\rho)} E^{\nu_\rho}[c_{0,e_i}(\eta) (\eta_0-\eta_{e_i})^2]
\qquad \text{for all } \rho \in (0,1).
$$ 
\end{lem}

Then, using Lemma \ref{lem:GK}, assumption (4) and \eqref{cxy:LB}, we obtain
\begin{multline} \label{Pp:pos}
  P'(\rho) 
  = \frac1{2\rho(1-\rho)} E^{\nu_\rho}[c_{0,e_i}(\eta) (\eta_0-\eta_{e_i})^2]
  = E^{\nu_\rho}[c_{0,e_i}(\eta)] \geq c_*
  \\ \text{for all $\rho \in (0,1)$ and } i = 1,\ldots,d.
\end{multline}

In addition to Assumption (6) on $P_i$, we assume on $f$ that it is bistable and $P$-balanced:
\begin{itemize}
\item[(BS)] $f$ has precisely three zeros $\a_1, \a_*, \a_2$ such that $0<\a_1< \a_*< \a_2<1$,
$f'(\a_1)<0$, $f'(\a_*)>0$ and  $f'(\a_2)<0$.  Also,
$$
A(f) := \int_{\a_1}^{\a_2} f(\rho) P'(\rho) \, d\rho=0
$$
holds.
\end{itemize}
Thanks to this assumption, \eqref{eq:RD-b} is a nonlinear Allen-Cahn equation for which the solution $\rho$ at a given time $t$ is in value close to either the phase $\alpha_1$ or the phase $\alpha_2$, except for a transition layer of width $1/\sqrt K$ between the two phases. Example \ref{exa:4.1} demonstrates that, given any exchange rates $c_{x,y}$ satisfying (1)--(6), it is possible to construct flip rates $c(\eta)$ (which only depend on the values of $\eta$ at three sites) for which the corresponding $f$ satisfies (BS).  

\begin{exa}  \label{exa:4.1}
Fix $x, y \in \Z^d$ such that $x, y, 0$ are mutually different. Let $c$ be as in \eqref{eq:c+-} with
\begin{align*} 
& c^+(\eta) = a_1 \eta_{x}\eta_{y}+ a_2 \eta_{x} + a_3 \geq 0, \\
& c^-(\eta) = b_1 \eta_{x}\eta_{y}+ b_2 \eta_{x} + b_3 \geq 0.
\end{align*}
Then,
\begin{align*} 
f(\rho) & = -(a_1 + b_1)\rho^3 + (a_1 - a_2 - b_2)\rho^2 +(a_2 - a_3 - b_3)\rho + a_3 \\
& = - (\rho - \a_1)(\rho-\a_*)(\rho-\a_2).
\notag
\end{align*}
with $0<\a_1<\a_*<\a_2<1$ if, for instance,
\begin{align*}
   a_1 + b_1 &=1, \\
   a_1 - a_2 - b_2 &= \a_1+\a_*+\a_2, \\
    a_2 - a_3 - b_3 &= -(\a_1\a_*+\a_*\a_2+\a_2\a_1), \\
a_3 &= \a_1\a_*\a_2.
\end{align*} 
In particular, $f(\rho)$ is bistable with stable points $\rho= \a_1, \a_2$
and unstable point $\a_*$.  To demonstrate how such an $f$ can be constructed, we take $a_1=1$, $a_2=0$, $b_1=0$ and
$\a_1=\frac14$, $\a_2=\frac34$.  Then, the conditions for $a_3$, $b_2$ and $b_3$ are
$a_3 = 3\a_*/16$, $b_2 = -\a_*$ and $b_3 = (13\a_*+3)/16$, where $\a_* \in (\a_1, \a_2)$ remains free to be chosen. For any such $\a_*$, it is 
 easy to check that $c^+(\eta)>0$ and $c^-(\eta)>0$.
To find $\a_*$ for which the $P$-balance condition holds, we observe from \eqref{Pp:pos} that
$A(f)>0$ if $\a_*$ is close to $\a_1$,
while $A(f)<0$ if $\a_*$ is close to $\a_2$.  Therefore, there exists an $\a_* \in (\a_1, \a_2)$
such that $A(f)=0$ holds.
\end{exa}

\subsection{Main result: the hydrodynamic limit}
\label{s:intro:MainThm}

The precise statement of the hydrodynamic limit of the process $\eta^N(t)$ generated by $L_N$ resembles \cite[Theorem 2.1]{EFHPS}. It requires some preparation.

First, we introduce a topology in which $\eta^N(t)$ can be compared to a solution $\Gamma_t$ of the curvature flow \eqref{MCF}. With this aim we construct different representations $\alpha^N(t)$ and $\chi_{\Gamma_t}$ for $\eta^N(t)$ and $\Gamma_t$ respectively. To any configuration $\eta \in \mathcal X_N$ we associate the macroscopically scaled empirical measure $\alpha^N$ on $\T^d$ defined by
\begin{align*}
& \a^N(dv;\eta) = \frac1{N^d} \sum_{x\in \T_N^d}
\eta_x \de_{x/N}(dv),\quad v \in \T^d.
\end{align*}
For the process $\eta^N(t)$, we set
\begin{align*}
& \a^N(t,dv) = \a^N(dv;\eta^N(t)), \quad t\ge 0.
\end{align*}
To any closed hypersurface $\Gamma_t$ we associate the piecewise constant function $\chi_{\Gamma_t}$ on $\T^d$ defined by
\begin{equation*}
  \chi_{\Gamma_t} (v)
  = \begin{cases}
    \alpha_1
    &\text{for $v$ on one side of $\Gamma_t$,}  \\
    \alpha_2
    &\text{for $v$ on the other side of $\Gamma_t$,}
  \end{cases}
\end{equation*}
where $\alpha_1, \alpha_2$ are the values introduced in (BS). 
The choice on the side of $\Gamma_t$ which is associated to $\alpha_1$ will be determined initially and kept fixed during the evolution.

Second, we define the sense in which the initial conditions $\eta^N(0)$ and $\Gamma_0$ have to be close to one other. For technical reasons, we cannot simply require that $\alpha^N(0)$ converges in a certain sense to $\chi_{\Gamma_0}$. Instead, the initial distribution $\mu^N$ of $\eta^N(0)$ on $\mathcal X_N$ needs to correspond to some sufficiently regular  function $u_0 : \T^d \to [0,1]$. To make this precise, we introduce further notation. To $u_0$ we associate on the microscopic scale the function 
\begin{equation} \label{u0N}
  u_0^N(x) := u_0(x/N)
  \qquad \text{for all } x \in \T_N^d
\end{equation}
and the (inhomogeneous) product measure 
\begin{equation} \label{nu0N}
  \nu_0^N := \nu_{u_0^N(\cdot)} , 
\end{equation}
where 
\begin{equation*}
  \nu_{u(\cdot)}(\eta) := \prod_{x \in \T_N^d} \nu_{u(x)} (\eta_x)
\end{equation*}
defines a product measure on $\mathcal X_N$ for any function $u : \T_N^d \to [0,1]$. 
Here, note that $\nu_{u(x)}$ is a measure on $\{0,1\}$ defined only on a single site $x$ with mean $u(x)$.
Finally, for two measures $\mu, \nu$ on $\mathcal X_N$ with $\nu$ having full support, we define the relative entropy as
$$
H(\mu|\nu) := \int_{\mathcal{X}_N} \frac{d\mu}{d\nu} \log \frac{d\mu}{d\nu} \, d\nu.
$$ 
The list of conditions on the initial data $\mu^N$ and $\Gamma_0$ are:
\begin{itemize}  
  \item[(IC1)] $\Gamma_0$ is a closed hypersurface of class $C^{5+\vartheta}$ for some $\vartheta \in (0,1)$. 
  \item[(IC2)] $H(\mu^N | \nu_0^N) = O(N^{d - \epsilon})$ as $N \to \infty$ for some $\epsilon > 0$, where $\nu_0^N$ is defined in \eqref{nu0N} in terms of some $u_0 \in C^{5}(\T^d)$ which satisfies:
  \begin{itemize}
      \item $\Gamma_0 = \{ v \in \T^d : u_0(v) = \alpha_* \}$, where $\alpha_*$ is defined in (BS),
      \item $\nabla u_0 (v) \cdot \mathbf{n}(v) \neq 0$ at any $v \in \Gamma_0$, 
      where $\mathbf{n}(v)$ is the normal direction to $\Ga_0$ at $v$,
      and
      \item there exist $0 < u_- < u_+ < 1$ such that $u_- \leq u_0(v) \leq u_+$ for all $v \in \T^d$.
    \end{itemize}  
\end{itemize}

Third, given an initial distribution $\mu^N$, we denote by $\mathbb P^{\mu^N}$ and $\E^{\mu^N}$ the process measure and expectation with respect to $\eta^N(\cdot)$.

\begin{thm} [Hydrodynamic limit] \label{thm:main}
Let $d \geq 2$. Let $\Gamma_t$ be the solution of \eqref{MCF} with initial condition $\Gamma_0$, and let $T > 0$ be such that $\Gamma_t$ is smooth for all $0 < t \leq T$. Let $L_N$ be as defined in Section \ref{s:intro:GK} under assumptions {\rm (1)--(6)} and {\rm (BS)}, and let $\eta^N(t)$ be the process generated by $L_N$ with initial distribution $\mu^N$. Assume that $\mu^N$ and $\Gamma_0$ satisfy {\rm (IC1)} and {\rm (IC2)}. Then, there exist constants $\sigma \in (0,1)$ and $\delta > 0$ such that if $K = K(N) \to \infty$ as $N \to \infty$ with $K \leq \delta (\log N)^{\sigma/2}$, then
\begin{equation*} 
   \lim_{N \to \infty} \mathbb P^{\mu^N} \big( \big| \langle \alpha^N(t), \phi \rangle - \langle \chi_{\Gamma_t}, \phi \rangle \big| > \e \big) = 0
 \end{equation*} 
 for all $t \in (0, T]$, all $\e > 0$ and all $\phi \in C^\infty(\T^d)$.
\end{thm}

We remark that, while here the constant $\sigma \in (0,1)$ is implicit, it is a H\"older exponent determined by a Nash estimate in \cite{FS}, which only involves $d, P, f$;
see Section \ref{s:H:uN:bds}.

\subsection{Proof outline for the hydrodynamic limit}
\label{s:intro:BG}

The proof of Theorem \ref{thm:main} is divided into several steps. Most of these steps are standard, proven in \cite{EFHPS} or proven elsewhere. The new step is the proof of a Boltzmann-Gibbs principle (Theorem \ref{BG}). In Section \ref{s:intro:BG:mot} we demonstrate how this Boltzmann-Gibbs principle is the essential new result on which Theorem \ref{thm:main} is built. Then, in Section \ref{s:intro:BG:pf} we summarize how Theorem \ref{thm:main} follows from this Boltzmann-Gibbs principle, and focus on the comparison of our proof to that  in \cite{EFHPS}.

\subsubsection{Narrowing down to a Boltzmann-Gibbs principle}
\label{s:intro:BG:mot}

As remarked in the introduction, the main result \cite[Theorem 2.1]{EFHPS} is the counterpart of Theorem \ref{thm:main} when the Kawasaki dynamics are replaced with the zero-range process. In particular, the limiting equation \eqref{MCF} is of the same form. In fact, we can use a similar breakdown of the proof of Theorem \ref{thm:main} as for the proof of \cite[Theorem 2.1]{EFHPS}. 
  
As in \cite{EFHPS}, let $u^N(t) = \{u^N(t,x)\in [0,1]\}_{x\in \T_N^d}$
be the solution of the discretized version of \eqref{eq:RD-b} given by
\begin{equation} \label{eq:HD-discre}
\partial_t u^N(t,x) =  \De^N P(u^N(t))(x) + K f(u^N(t,x))
\end{equation}
with initial condition $u^N(0,x) = u_0^N(x)$ (recall \eqref{u0N}), where 
\begin{align*}
  \De^N u(x) 
  &:= N^2 \sum_{i=1}^d\Big(u(x+e_i) + u(x-e_i) - 2u(x)\Big) 
  = N \sum_{y : |y-x| = 1} \nabla_{yx}^N u, \\
  \nabla_{xy}^N u &:= N \nabla_{xy} u, \quad 
  \nabla_{xy}u := u(x) - u(y).
\end{align*} 
For later use, we also set
\[
  \nabla^N u(x) := \{ \nabla_{x+e_i, x}^N u \}_{i=1}^d = \{N(u(x+e_i)-u(x))\}_{i=1}^d.
\] 
We interpret $u^N(t,x)$ as an approximation of $\E^{\mu^N} [\eta_x^N(t)]$.

It is shown in \cite{EFHPS} (see also Section \ref{s:H:uN:bds}) that the bounds in (IC2) imply $0 < u^N(t,x) < 1$ for any $t\in[0,T]$ and any $x\in\T_N^d$. Then, similar to \eqref{nu0N}, we can take
\begin{equation*} 
  \nu_t^N := \nu_{u^N(t, \cdot)}.
\end{equation*}
Moreover, \cite[Theorem 2.3]{EFHPS} gives a precise statement on the pointwise convergence of $u^N(t)$ (when interpreted as a piecewise constant function on the macroscopic domain $\T^d$) to $\chi_{\Gamma_t}$ as $N \to \infty$. Then, the argument used in \cite[Section 2.3]{EFHPS} narrows the proof of Theorem \ref{thm:main} down to the entropy estimate
\begin{equation} \label{H:est:intro}
  H(\mu_t^N | \nu_t^N) = o(N^d)  
  \quad \text{as $N \to \infty$ for all $t \in (0,T]$},  
\end{equation}
where $\mu_t^N$ is the law of the process $\eta^N(t)$. 

Establishing an entropy estimate of the type \eqref{H:est:intro} is of interest beyond the proof of Theorem \ref{thm:main} to the study of fluctuations around the hydrodynamic limit; see e.g.\ \cite{JM1,JM2}. Therefore, we give a precise statement for this estimate (Theorem \ref{thm:EstHent}, although not as sharp as in \cite{JM1,JM2}), and state it in more generality than required for the proof of Theorem \ref{thm:main}.

\begin{thm}[Entropy estimate] \label{thm:EstHent}
Let $d \geq 1$, $T > 0$ and $K = K(N)$ be increasing in $N$. Let $\eta^N(t)$ be as in Theorem \ref{thm:main}. Let $\mu_t^N$, $u^N(t,x)$ and $\nu_t^N$ be as above. If there exist constants $C_0, C_1, \epsilon > 0$ and $0<u_-<u_+<1$ such that for all $N$ large enough
\begin{enumerate} [label=\rm{(\roman*)}]
  \item \label{thm:EstHent:u:LBUB} $u_- \le u^N(0,x) \le u_+$
for all $x\in \T_N^d$, 
  \item \label{thm:EstHent:gradu:UB} $|\nabla_{y_1,x}^N \cdots \nabla_{y_k,x}^N u^N(0,x)|
\le C_0$ for all $1\le k \le 4$ and all $x, y_1,\ldots, y_4 \in \T_N^d$ such that $|y_k-x|=1$ for all $1\le k \le 4$, and
  \item $H(\mu^N|\nu_0^N) \leq C_1 N^{d-\epsilon}$,
\end{enumerate} 
then there exist $\delta > 0$ and $\sigma \in (0,1)$ such that if $1\le K \le \de (\log N)^{\si/2}$ for all $N$ large enough, then
$$
H(\mu_t^N|\nu_t^N) = o(N^d) \quad \text{as $N\to\infty$ for all } t\in [0,T].
$$
\end{thm} 

We comment on the relaxation of conditions (IC1) and (IC2). (IC1) is not required in Theorem \ref{thm:EstHent}.
The new conditions (i)--(iii) resemble those of (IC2), but are weaker, because $u_0^N$ need not be related by \eqref{u0N} to some regular function $u_0$. 

We prove Theorem \ref{thm:EstHent} in Section \ref{s:H}. In this proof we use Yau's relative entropy method \cite{Y} to compute and bound from above the entropy production $\frac d{dt} H(\mu_t^N|\nu_t^N)$, and then conclude by applying Gronwall's lemma. In this computation, terms of the type
\begin{equation} \label{need:for:BG}
  \E^{\mu^N} \bigg[\int_0^T \sum_{x \in \T_N^d} a_{t,x} \tau_x h(\eta(t)) dt \bigg]
\end{equation}
appear, where we recall that $\E^{\mu^N}$ is the expectation with respect to the process $\eta^N(\cdot)$ with initial distribution $\mu^N$, and where $h$ is a local function on $\mathcal X_N$ which may be different for each term of type \eqref{need:for:BG}. Each function $h$ does not depend on $t,x,N$; it is expressed in terms of the flip rate $c$, one of the exchange rates $c_{0,e_i}$, or one of the functions $h_i$ from the gradient condition (5). The coefficients $a_{t,x}$ may also differ from term to term. They are deterministic and explicitly expressed in terms of $u^N(t,x)$. Hence, $a_{t,x} =a_{t,x}(N)$ may depend on $N$. The difficulty in dealing with \eqref{need:for:BG} is that we have little control over $\mu_t^N$. We overcome this issue by establishing a Boltzmann-Gibbs principle, which allows us to replace $\tau_x h(\eta(t))$ by averages over $h$ such that the only randomness left is through an affine dependence on $\eta_x(t)$. The cost of this replacement is an entropy term and a deterministic error term, which are sufficiently small for the application of Gronwall's lemma. 

Next we prepare to state the Boltzmann-Gibbs principle, Theorem \ref{BG}. As for Theorem \ref{thm:EstHent}, the Boltzmann-Gibbs principle is also of interest in the study of fluctuations around the hydrodynamic limit; see e.g.\ the book chapters \cite[{Chapter} 7]{KLO} 
and \cite[{Chapter} 11]{KL}.
Therefore, we state a new {\it time-dependent} version in terms of a parameter $\Theta$, which is more general than strictly required for the proof of Theorem \ref{thm:main}. 

Let $K = K(N) \geq 1$ be an increasing function. Instead of a solution $u^N(t,x)$ to \eqref{eq:HD-discre}, we consider deterministic values $u_{t,x} = u_{t,x}(N)$ which satisfy the uniform bound in Theorem \ref{thm:EstHent}\ref{thm:EstHent:u:LBUB}  uniformly in $t$ and, for some constant $\Theta \in [0,1)$, the bound on the gradient 
\begin{equation}  \label{eq:nablau}
|u_{t,y} - u_{t,x}| \le C \frac{K^{1/\si}}{N [t(T-t)]^\Theta}
\end{equation} 
for all  $t \in [0,T]$, all  $x,y \in \T_N^d$  with  $|x-y| = 1$ and all $N$ large enough. Let $\nu_t^N$ be the Bernoulli measure with mean $\{ u_{t,x} \}_{x \in \T_N^d}$.  For a local function $h=h(\eta)$ with support in a finite square box 
$\Lambda_h\subset \T_N^d$, let  
\begin{align}
\label{f_tx eq}
  f_{t,x}(\eta) := \tau_xh(\eta) -\tilde h(u_{t,x}) -\tilde h'(u_{t,x})\big(\eta_x - u_{t,x}\big),
\end{align}
where
\begin{align}
\label{ht} 
  \tilde h(\beta) := E^{\nu_\beta}[h]
\end{align} 
for any $\beta\in [0,1]$; recall that $\nu_\b$ is the Bernoulli measure with mean $\b$.
Let $\mu^N=\mu_0^N$ be an initial distribution on $\mathcal X_N$, and $\mu^N_t$ be the 
distribution of $\eta^N(t)$ generated by $L_N=N^2L_K+KL_G$.

\begin{thm}[Boltzmann-Gibbs principle]
\label{BG}
Let $d \geq 1$ and $T > 0$. Let $\eta^N(t)$ be as in Theorem \ref{thm:main}. Let $\mu_t^N$ be the law of $\eta^N(t)$.
Let $K$, $\Theta$, $u_{t,x}$, $\nu_t^N$, $h$ and $f_{t,x}$ be as above. Let $a_{t,x} = a_{t,x}(N) \in \R$ for $t \in (0,T)$ and $x \in \T_N^d$ be deterministic coefficients which satisfy
\begin{equation} \label{BG:assn:atx}  
  \sup_{x \in \T_N^d} |a_{t,x}| \leq C \frac{ K^\theta }{ [t (T-t)]^\Theta } \quad \text{for all } t \in (0,T)
\end{equation}
for some constant $\theta \geq 0$.
Then there exist  $C, \epsilon_0, \epsilon_1 >0$ such that for all $N$ large enough, if $K = O (N^{\epsilon_0})$ as $N \to \infty$, then
\begin{equation} \label{BG:est}
\E^{\mu^N} \left |\int_0^T \sum_{x \in \T_N^d} a_{t,x} f_{t,x} dt\right | 
\leq C \bigg( K^\theta \int_0^T  H(\mu_t^N| \nu_t^N) \frac{dt}{ [t(T-t)]^\Theta } + N^{d-\epsilon_1} \bigg).
\end{equation}
\end{thm}

We prove Theorem \ref{BG} in Section \ref{s:BG:pf}. Our proof gives an explicit construction for the constants $\epsilon_0$ and $\epsilon_1$.

When we apply Theorem \ref{BG} in the proof of Theorem \ref{thm:EstHent}, we do not need the absolute value signs in \eqref{BG:est}. Also, we only apply Theorem \ref{BG} with $\Theta = 0$. Then, the term $[t(T-t)]^\Theta$ disappears from the denominator in \eqref{BG:est}, and it is possible to select larger values for $\epsilon_0$ and $\epsilon_1$.

\subsubsection{Outline of the proofs of Theorems \ref{thm:main}, \ref{thm:EstHent} and \ref{BG}}
\label{s:intro:BG:pf}

We have argued that the key step in the proof of our main Theorem \ref{thm:main} is the Boltzmann-Gibbs principle in Theorem \ref{BG}. Here, we give a more precise overview of the proof, and focus on the comparison with \cite{EFHPS}.

Given the entropy estimate in Theorem \ref{thm:EstHent}, the proof of Theorem \ref{thm:main} is given in \cite{EFHPS}. Indeed, as shown in \cite[Section 2.3]{EFHPS}, the convergence in probability stated in Theorem \ref{thm:main} is a direct consequence of:
\begin{itemize}
   \item the approximation of the expectation of the process $\eta^N(t)$ by the discrete density $u^N$ (see Theorem \ref{thm:EstHent}), and
   \item the convergence of $u^N(t)$ to $\chi_{\Gamma_t}$ as $N \to \infty$ (see \cite[Theorem 2.3]{EFHPS} for a precise statement).
 \end{itemize}   
This concludes the proof of Theorem \ref{thm:main} in terms of Theorem \ref{thm:EstHent}.
\smallskip

The outline of the proof of Theorem \ref{thm:EstHent} is briefly given above in the paragraph which contains \eqref{need:for:BG}. Here, we provide more details on the computation (and estimate from above) of $\frac d{dt} H(\mu_t^N|\nu_t^N)$. The full computation is done in Section \ref{s:H}. The goal of the computation is to bound $\frac d{dt} H(\mu_t^N|\nu_t^N)$ from above by a sum of several terms depending on $u^N$, and to collect the terms which are asymptotically the largest in $N$. For these terms to cancel each other out, it is necessary that the density $u^N$ satisfies Equation \eqref{eq:HD-discre}. To control the remaining terms, which are of higher order, we need sufficient bounds on the first and second order gradients of $u^N(t)$. We obtain these bounds ultimately from Assumption \ref{thm:EstHent:gradu:UB} in Theorem \ref{thm:EstHent}. 

Also in \cite{EFHPS} the entropy production $\frac d{dt} H(\mu_t^N|\nu_t^N)$ is computed. The computation is similar for the terms related to $\partial_t u^N$ and $K f(u^N)$, but the most involved part, i.e.\ the computation of the nonlinear diffusion term $\Delta^N P(u^N)$, is completely different. Indeed, we strongly rely on the assumptions (2),(4),(5) on the exchange rates $c_{x,y}(\eta)$ to rewrite the contribution of $L_K$ to the entropy production as a sum of terms of the form \eqref{need:for:BG}. 
\smallskip

Finally, we turn to the proof of Theorem \ref{BG}. The first step is the same as in \cite{EFHPS}: we replace $f_{t,x}$ in \eqref{BG:est} by the conditional expectation with respect to a local average over $\eta$ around $x$, and estimate separately the error made by this replacement. We estimate this error in a similar manner as in \cite[Section 7]{EFHPS} (see Lemma \ref{l:L10:6} below), but with several modifications to account for the Kawasaki process. However, for the term involving the conditional expectation, the finite-range particle interactions of the Kawasaki process requires a different approach; see Lemma \ref{l:L10:9} below. The proof of this lemma is not based on \cite{EFHPS}. Instead, it is a generalization of the argument used in \cite{FT}, which was originally designed for the simple Kawasaki process.

\subsection{Discussion}
\label{s:intro:disc}

\paragraph{On the entropy estimate.} Also the entropy estimate, Theorem \ref{thm:EstHent}, is stated in higher generality. Indeed, $u^N$ need not be defined through a macroscopic profile as in \eqref{u0N}, and its discrete derivatives need not be uniformly bounded in $N$; see Theorem \ref{thm:EstHent}\ref{thm:EstHent:gradu:UB}. Yet, we expect that a weaker bound than that in Theorem \ref{thm:EstHent}\ref{thm:EstHent:gradu:UB} will be sufficient. Indeed, we only use Theorem \ref{thm:EstHent}\ref{thm:EstHent:gradu:UB} to obtain bounds on the discrete derivatives of $u^N(t,x)$ in \eqref{uN:der:bds} derived in \cite{FS} which are uniform in time. Without any bounds on the derivatives of $u^N(0,x)$, \cite[Corollaries 4.3 and 5.8]{FS} guarantees that
\begin{equation} \label{uN:bds:intro}
  |\nabla^N u^N(t,x)| \leq \frac{ C K^{1/\sigma} }{\sqrt t} 
  \quad \text{and} \quad
  |\Delta^N u^N(t,x)| \leq \frac{  C K^{2/\sigma} }t.
\end{equation}
The singularity of these bounds at $t=0$ would correspond in our use of Theorem \ref{BG} to $\Theta = 1$, which is the critical value at which this theorem does not apply. This gives reason to expect that under weaker conditions than those in Theorem \ref{thm:EstHent}\ref{thm:EstHent:gradu:UB} a certain bound on $u^N$ holds (weaker than \eqref{uN:der:bds} but stronger than \eqref{uN:bds:intro}) for which Theorem \ref{BG} applies. We leave this for future research.

\paragraph{On the assumptions in Theorem \ref{thm:main}.} We wish to comment on a few assumptions in Theorem \ref{thm:main} which may be lifted in future works. First, it may be possible to remove assumption (6). If so, we expect \eqref{eq:RD} to be the hydrodynamic limit in the case where $K$ is constant. To consider the case $K \to \infty$ as $N \to \infty$, the following two preliminary problems should be addressed first:
\begin{enumerate}
\item the sharp-interface limit for \eqref{eq:RD}, which is a PDE-theory problem,
\item convergence of the solution of the discretized version of \eqref{eq:RD}, and bounds on its discrete derivatives in terms of $K$ and $N$.
\end{enumerate}
These problems are interesting on themselves.

Second, it is desired to remove the gradient condition (5) in addition to removing (6). Then, for constant $K$ we expect 
\[
  \partial_t\rho= \sum_{i,j=1}^d \frac{\partial}{\partial v_i} \left(D_{ij}(\rho)\frac{\partial\rho}{\partial v_j}\right) + K f(\rho), \quad v\in \T^d
\]
to be the hydrodynamic equation, where $D(\rho)=\{D_{ij}(\rho)\}_{i,j=1}^d$ is the diffusion matrix of density $\rho\in[0,1]$. For the case $K \to \infty$ as $N \to \infty$ we expect that the equivalent of Theorem \ref{thm:EstHent} is very challenging to prove; we refer to \cite{FUY} and \cite[{Chapter} 7]{KL} for hydrodynamic limits of simpler non-gradient systems.

Third, it is desirable to weaken the bound on $K$ from $\log N$ to $N^\epsilon$. The relative entropy method does not seem fit for this challenge. While the method based on correlation functions as used in \cite{KS,Bon} allows for power law bounds on $K$ in simpler scenarios, it seems difficult to extend it to our setting in which $c_{x,y}(\eta)$ depends on $\eta$. 

%
%

\paragraph{Organization of the paper.} In Section \ref{s:H} and Section \ref{s:BG:pf} we prove Theorem \ref{thm:EstHent} (the entropy estimate) and Theorem \ref{BG} (the Boltzmann-Gibbs principle) respectively. 

\section{Proof of Theorem \ref{thm:EstHent}}  
\label{s:H}

As mentioned in the introduction, we will prove Theorem \ref{thm:EstHent} by applying Gronwall's lemma to the relative entropy $H(\mu_t^N|\nu_t^N)$. With this aim, we first establish bounds on $u^N(t,x)$ in the preliminary Section \ref{s:H:uN:bds}. Then, we compute $\frac d{dt} H(\mu_t^N|\nu_t^N)$ in Section \ref{s:H:ddt}. To bound the terms in the result of this computation, we apply the Boltzmann-Gibbs principle, Theorem \ref{BG}, in Section \ref{s:H:BG} to complete the estimate to which Gronwall's lemma applies.

\subsection{Bounds on $u^N$}
\label{s:H:uN:bds} 

\paragraph{Maximum principle for the discrete hydrodynamic equation
\eqref{eq:HD-discre}.}
Since $P$ is monotone, $f$ is bistable and $u^N(0,x)$ is bounded from below and above (see Theorem \ref{thm:EstHent}\ref{thm:EstHent:u:LBUB}), it follows from the maximum principle of \eqref{eq:HD-discre} (see \cite[Lemma 3.1]{EFHPS}) that
\begin{equation} \label{uN:uf:bds}
  0 < \min \{ u_-, \alpha_1 \}
  \leq u^N(t,x)
  \leq \max \{ u_+, \alpha_2 \} < 1
  \quad \text{for all } t \in [0,T] \text{ and all } x \in \T_N^d,
\end{equation}
where $\alpha_1$ and $\alpha_2$ are the stable points of $f$ (see (BS)).

\paragraph{Uniform bounds on $\nabla^N u^N$ and $\Delta^N u^N$.} In \cite{FS} uniform bounds are constructed for discrete derivatives of the solution $u^N$ to \eqref{eq:HD-discre}. For our purpose, the bounds provided in \cite[Corollaries 4.4 and 5.10]{FS} are sufficient. These bounds state that if $u^N(0,x)$ satisfies the bounds in Theorem \ref{thm:EstHent}\ref{thm:EstHent:gradu:UB}, then 
\begin{subequations} \label{uN:der:bds}
\begin{align} \label{uN:der:bds:1}
  |\nabla^N u^N(t,x)| &\leq C K^{1/\sigma},  \\\label{uN:der:bds:2}
  |\Delta^N u^N(t,x)| &\leq C K^{2/\sigma}
\end{align}
\end{subequations}
for all $t \in [0,T]$, $x \in \T_N^d$ and $K\ge 1$, where $C > 0$ is a constant independent of $N, K, t, x$, and $\sigma \in (0,1)$ is the exponent introduced in Theorem \ref{thm:main}, depending only on $d$, $P$ and $f$.

\subsection{Time derivative of relative entropy}
\label{s:H:ddt}

We start by citing a theorem on a bound for $\frac d{dt} H(\mu_t^N|\nu_t)$, where $\nu_t$ is any probability measure on
$\mathcal{X}_N$ with full support and differentiable in $t$. With this aim and for later use, we introduce some notation.
Let $\bf 1$ be the constant function equal to $1$ on $\mathcal{X}_N$ and $\bf 1_{\{\eta \in \mathcal A\}}$ be the indicator function for some $\mathcal A \subset \mathcal{X}_N$.
For a function $f$ on $\mathcal{X}_N$ 
and a measure $\nu$ on $\mathcal{X}_N$, set
\begin{align*} 
\mathcal{D}_N(f;\nu) = 2N^2 \mathcal{D}_K(f;\nu) + K \mathcal{D}_G(f;\nu)
\end{align*}
as the Dirichlet form, where
\begin{align*}
\mathcal{D}_K(f;\nu) & = \frac14 \sum_{x,y\in\T_N^d} \int_{\mathcal{X}_N} 
 c_{x,y}(\eta) \{f(\eta^{x,y})-f(\eta)\}^2 d\nu,  \\
\mathcal{D}_G(f;\nu) & = \sum_{x\in\T_N^d} \int_{\mathcal{X}_N} 
 c_x(\eta)\{f(\eta^x)-f(\eta)\}^2 d\nu.
\end{align*}
Take $m$ as a reference measure on 
$\mathcal{X}_N$ with full support, and set $\psi_t := \frac{d\nu_t}{dm}$.  Finally,
We denote the adjoint of an operator $L$ on $L^2(m)$ by $L^{*,m}$. Then, we have the following result. A proof can be found in \cite[{Theorem} 4.2]{F18} or \cite[{Lemma} A.1]{JM2}.

\begin{thm}[Upper bound on entropy production] \label{thm:4.2} 
\begin{equation} \label{eq:dH}
\frac{d}{dt} H(\mu_t^N|\nu_t) \le - \mathcal{D}_N\left(\sqrt{\frac{d\mu_t^N}{d\nu_t}}; \nu_t\right) 
+ \int_{\mathcal{X}_N} (L_N^{*,\nu_t}{\bf 1} - \partial_t \log \psi_t) d\mu_t^N.
\end{equation}
\end{thm}

We apply Theorem \ref{thm:4.2} with $\nu_t = \nu_t^N$. We will not use $\mathcal{D}_N$ to derive our bounds; we simply estimate $\mathcal{D}_N \geq 0$.
We split the second term in the right-hand side of \eqref{eq:dH} in three parts; one related to $L_K$, one to $L_G$ and one to $\psi_t$, and treat all computations separately in the three lemmas below.

For the computations related to $L_N$ (recall \eqref{LN}), the time dependence is irrelevant. In these computations, we drop $t$ from the notation whenever convenient. Moreover, we set 
\begin{align} \label{etab}
u_x = u^N(t,x), \qquad
\eta_x = \eta_x(t), \qquad
\bar\eta_x = \eta_x-u_x
\end{align}
and define the compressibility
\begin{equation} \label{chi}
  \chi(\rho) = \rho(1-\rho) \quad \text{for } \rho \in [0,1].
\end{equation}

\begin{lem}  \label{lem:4.4} 
\begin{align} \label{eq:3.3-b}
L_K^{*,\nu_t^N}{\bf 1} =  & -\frac12 \sum_{x,y\in\T_N^d:|x-y|=1}  \frac{(u_y-u_x)^2}{\chi(u_x) \chi(u_y)} c_{x,y}(\eta) \bar\eta_x \bar\eta_y\\\notag
& + \frac12 \sum_{x, y\in\T_N^d:|x-y|=1} c_{x,y}(\eta) \Big( \frac{ \bar\eta_x }{ \chi(u_x) } - \frac{ \bar\eta_y }{ \chi(u_y) } \Big)(u_y-u_x).\notag
\end{align}
\end{lem}

\begin{proof}
Recall
$$
\int_{\mathcal{X}_N} \Big( L_K^{*,\nu_t^N}{\bf 1} \Big)  f \, d\nu_t^N
= \int_{\mathcal{X}_N} L_K f \, d\nu_t^N 
= \frac12 \sum_{x,y} \sum_\eta c_{x,y}(\eta)
\{f(\eta^{x,y})-f(\eta)\} \nu_t^N(\eta)
$$
for any function $f$ on $\mathcal{X}_N$,
where $\sum_{x,y} =\sum_{x,y\in\T_N^d: |x-y|=1}$ and $\sum_\eta = \sum_{\eta \in \mathcal{X}_N }$.
By the change of variables $\eta\mapsto\eta^{x,y}$ and condition (4) of 
$c_{x,y}(\eta)$, we obtain
\begin{align*}  
 \sum_{x,y} & \sum_\eta c_{x,y}(\eta) f(\eta^{x,y})\nu_t^N(\eta)
 = \sum_{x,y} \sum_\eta c_{x,y}(\eta)f(\eta)\nu_t^N(\eta^{x,y}) \\
& = \sum_{x,y} \sum_\eta c_{x,y}(\eta)f(\eta) \bigg( \frac{(1-u_x)u_y}{u_x(1-u_y)} {\bf 1}_{\{\eta_x=1,\eta_y=0\}}
+\frac{(1-u_y)u_x}{u_y(1-u_x)} {\bf 1}_{\{\eta_y=1,\eta_x=0\}} \\
& \hskip 50mm + {\bf 1}_{\{\eta_x=\eta_y\}} \bigg) \nu_t^N(\eta) \\
& =  \sum_{x,y} \sum_\eta \left( c_{x,y}(\eta)f(\eta) \nu_t^N(\eta) 
+ c_{x,y}(\eta)f(\eta) \frac{u_y-u_x}{u_x(1-u_y)} {\bf 1}_{\{\eta_x=1,\eta_y=0\}}
 \nu_t^N(\eta) \right. \\
& \hskip 30mm  \left.
+ c_{x,y}(\eta) f(\eta) \frac{u_x-u_y}{u_y(1-u_x)} {\bf 1}_{\{\eta_y=1,\eta_x=0\}}
 \nu_t^N(\eta)\right).
\end{align*}
Since the second and third terms coincide by exchanging the role of
$x$ and $y$, and 
\begin{align*}  
{\bf 1}_{\{\eta_x=1,\eta_y=0\}} & = \eta_x(1-\eta_y) =
 (\bar\eta_x + u_x) (1-\bar\eta_y-u_y) \\
& = \bar\eta_x(1-\bar\eta_y) -\bar\eta_x u_y - \bar\eta_y u_x + u_x(1 -u_y),
\end{align*}
we have
$$  
L_K^{*,\nu_t^N}{\bf 1} = \sum_{x,y} c_{x,y}(\eta)\frac{u_y-u_x}{u_x(1-u_y)} 
\big( \bar\eta_x(1-\bar\eta_y) -\bar\eta_x u_y - \bar\eta_y u_x + u_x(1 -u_y) \big).
$$
Note that the sum of the last term is equal to $\sum_{x,y}c_{x,y}(\eta)(u_y-u_x)=0$ since $c_{x,y} = c_{y,x}$ for any $x,y$.
By exchanging the role of $x$ and $y$ in the
third term, the sum of the second and third term equals
\begin{align*}
- \sum_{x,y} &c_{x,y}(\eta)\left\{ \frac{u_y-u_x}{u_x(1-u_y)} u_y +\frac{u_x-u_y}{u_y(1-u_x)} u_y 
\right\} \bar\eta_x \\
& = - \sum_{x,y} c_{x,y}(\eta)(u_y-u_x) \frac{(1-u_x)u_y-u_x(1-u_y)}{(1-u_x)u_x(1-u_y)} \bar\eta_x \\
& = - \sum_{x,y} c_{x,y}(\eta)\frac{(u_y-u_x)^2}{(1-u_x)u_x(1-u_y)} \bar\eta_x.
\end{align*}
Hence,
\begin{equation*} 
L_K^{*,\nu_t^N}{\bf 1} = \sum_{x,y} c_{x,y}(\eta) \bar\eta_x \left\{\frac{u_y-u_x}{u_x(1-u_y)} (1-\bar\eta_y)
- \frac{(u_y-u_x)^2}{(1-u_x)u_x(1-u_y)} \right\}.
\end{equation*}
Since the $\bar\eta_y$-independent part inside the braces equals
\begin{align*}
&(u_y-u_x) \left\{\frac{1}{u_x(1-u_y)} - \frac{(u_y-u_x)}{(1-u_x)u_x(1-u_y)} \right\}  \\
&\qquad  = (u_y-u_x) \frac{(1-u_x)-(u_y-u_x)}{(1-u_x)u_x(1-u_y)}
 = \frac{u_y-u_x}{\chi(u_x)},
\end{align*}
we obtain
\begin{equation*} 
L_K^{*,\nu_t^N}{\bf 1} = -\sum_{x,y} c_{x,y}(\eta)\frac{u_y-u_x}{u_x(1-u_y)} \bar\eta_x \bar\eta_y
+ \sum_{x} \frac{\bar\eta_x}{\chi(u_x)}  \sum_{y:|x-y|=1} c_{x,y}(\eta)(u_y-u_x).
\end{equation*}
Finally, both terms can be symmetrized in $x$ and $y$. Since this is immediate for the second term, we focus only on the first term:
\begin{align*}
-\sum_{x,y} &c_{x,y}(\eta)\frac{u_y-u_x}{u_x(1-u_y)} \bar\eta_x \bar\eta_y \\
&= -\frac12 \sum_{x,y} c_{x,y}(\eta)\left\{ \frac{u_y-u_x}{u_x(1-u_y)} 
+ \frac{u_x-u_y}{u_y(1-u_x)} \right\} \bar\eta_x \bar\eta_y \\
&=  -\frac12 \sum_{x,y} c_{x,y}(\eta)(u_y-u_x) \frac{u_y(1-u_x)- u_x(1-u_y)}
{\chi(u_x)\chi(u_y)} \bar\eta_x \bar\eta_y\\
&= -\frac12 \sum_{x,y}  c_{x,y}(\eta)\frac{(u_y-u_x)^2}{\chi(u_x)\chi(u_y)} 
\bar\eta_x \bar\eta_y.
\end{align*}
This completes the proof of \eqref{eq:3.3-b}.
\end{proof}

We have the following two lemmas for the Glauber part and $\partial_t \log\psi_t(\eta)$,
respectively.  

\begin{lem}[{\cite[{Lemma} 4.5]{F18}}] \label{lem:2.5FT}
Recalling the notation from \eqref{etab},
\begin{align} \label{eq:L_G1}
L_G^{*,\nu_t^N}{\bf 1} = \sum_{x\in\T_N^d} \left(\frac{ c_x^+(\eta)}{u_x}-
\frac{c_x^-(\eta)}{1-u_x}\right)\bar\eta_x.
\end{align}
\end{lem}

\begin{lem}[{\cite[{Lemma} 4.6]{F18}}] \label{lem:4.6F18}
Recalling the notation from \eqref{etab} and \eqref{chi}, and denoting explicitly the dependence on $t$,
\begin{equation} \label{eq:logpsi}
\partial_t \log\psi_t(\eta) = \sum_{x\in\T_N^d} \partial_t u_x(t)
\frac{ \bar\eta_x(t) }{ \chi(u_x(t)) }.
\end{equation}
\end{lem}

We remark that Lemmas \ref{lem:2.5FT} and \ref{lem:4.6F18} are proven by a direct computation in a similar style as for the proof of Lemma \ref{lem:4.4}.

We summarize in
the following corollary the identities \eqref{eq:3.3-b}, \eqref{eq:L_G1} and
\eqref{eq:logpsi} obtained in the three lemmas above.

\begin{cor}  \label{Cor:2.6}
Recalling the notation from \eqref{etab} and \eqref{chi}, we have
\begin{align}  \label{eq:cor2.6}
L_N^{*,\nu_t^N}{\bf 1} - \partial_t \log\psi_t(\eta) 
= & -\frac{N^2}2 \sum_{x,y\in \T_N^d:|x-y|=1} \frac{ (u_y(t)-u_x(t))^2 }{ \chi(u_x(t)) \chi(u_y(t)) } c_{x,y}(\eta) \bar\eta_x(t) \bar\eta_y(t) \\
&+\frac{N^2}2 \sum_{x, y\in\T_N^d:|x-y|=1} c_{x,y}(\eta) \Big( \frac{ \bar\eta_x(t) }{ \chi(u_x(t)) } - \frac{ \bar\eta_y(t) }{ \chi(u_y(t)) } \Big)(u_y(t)-u_x(t))  \notag\\
& + K \sum_{x\in\T_N^d} \left(\frac{ c_x^+(\eta)}{u_x(t)}-
\frac{c_x^-(\eta)}{1-u_x(t)}\right) \bar\eta_x(t)
- \sum_{x\in\T_N^d} \partial_t u_x(t) \frac{ \bar\eta_x(t) }{ \chi(u_x(t)) }.  \notag
\end{align}
\end{cor}

\subsection{Application of the Boltzmann-Gibbs principle}
\label{s:H:BG}

Here we prove Theorem \ref{thm:EstHent}. As in Section \ref{s:H:ddt} we set $u_x(t) = u^N(t,x)$. Let $T_1 \in (0,T]$. Using Theorem \ref{thm:4.2} and the bound on the relative entropy at $t = 0$, we obtain
\begin{align} \label{pf:3}
  H(\mu_{T_1}^N|\nu_{T_1}^N)
  & = \int_0^{T_1} \frac{d}{dt} H(\mu_t^N|\nu_t^N) \, dt + H(\mu^N|\nu_0^N) \\
  & \le \E^{\mu^N} \bigg[ \int_0^{T_1} (L_N^{*,\nu_t^N}{\bf 1} - \partial_t \log \psi_t) \, dt \bigg] + C N^{d - \epsilon}.  \notag
\end{align}

In what follows, we apply \eqref{eq:cor2.6} to the expectation in the right-hand side, and treat the resulting four terms separately. The fourth term needs no treatment; it is given by
\begin{equation*} 
  \sum_{x\in\T_N^d} \E^{\mu^N} \bigg[ \int_0^{T_1} - \partial_t u_x(t) \frac{ \bar\eta_x(t) }{ \chi(u_x(t)) } \, dt \bigg].
\end{equation*}
For the other three terms we apply the Boltzmann-Gibbs principle, Theorem \ref{BG}. Note from \eqref{uN:der:bds:1} that the requirement \eqref{eq:nablau} is satisfied. In the remainder, we will often use \eqref{uN:uf:bds} and \eqref{uN:der:bds:1} without explicit reference. We always apply Theorem \ref{BG} with $\Theta = 0$. The dependence on the time variable is of little importance in the computation below; we therefore drop it from the notation.

\paragraph{Third term of the right-hand side of \eqref{eq:cor2.6}.} First we rewrite the third term of the right-hand side of \eqref{eq:cor2.6} in a form to which we can apply Theorem \ref{BG}. With this aim, we expand the third term of the right-hand side of \eqref{eq:cor2.6} as 
\begin{align} \label{pf:5}
  &K \sum_{x\in\T_N^d} \left(\frac{ c_x^+(\eta)}{u_x}-
\frac{c_x^-(\eta)}{1-u_x}\right)\bar\eta_{x} \\\notag
&= K \sum_{x\in\T_N^d} \frac{ c_x^+(\eta) \eta_x }{u_x}
  - K \sum_{x\in\T_N^d} c_x^+(\eta)
  - K \sum_{x\in\T_N^d} \frac{c_x^-(\eta) \eta_x}{1-u_x}
  + K \sum_{x\in\T_N^d}  \frac{c_x^-(\eta) u_x}{1-u_x} \\\notag
  &=: I_1^+ + I_2^+ + I_1^- + I_2^-.
\end{align}
Each of the four terms $I_j^\pm$ ($j = 1,2$) can be written as
\[
  I_j^\pm = \sum_{x\in\T_N^d} a_{x}^{j, \pm} \tau_x h_j^\pm (\eta),
\]
where 
$$
h_1^\pm(\eta) := c^\pm(\eta) \eta_0, \qquad h_2^\pm(\eta) := c^\pm(\eta),
$$ and $$
a_{x}^{1, +} = \frac{K}{u_x}, \quad
a_{x}^{1, -} = \frac{-K}{1-u_x}, \quad
a_{x}^{2, \pm} = - u_x a_{x}^{1, \pm}.
$$
Note from \eqref{uN:uf:bds} that $|a_x^{j,\pm}| \leq CK$ for $j = 1,2$. Hence, $a_x^{j,\pm}$ and $h_j^\pm$ satisfy all conditions of Theorem \ref{BG}. Then, applying Theorem \ref{BG} with $\theta = 1$ to each of the four terms $I_j^\pm$, we obtain 
\begin{multline} \label{pf:4}
  \E^{\mu^N} \bigg[ \int_0^{T_1} K \sum_{x\in\T_N^d} \left(\frac{ c_x^+(\eta)}{u_x}-
\frac{c_x^-(\eta)}{1-u_x}\right)\bar\eta_{x} \, dt \bigg] \\
\leq \E^{\mu^N} \bigg[ \int_0^{T_1} \sum_{\pm} \sum_{j=1}^2 \sum_{x\in\T_N^d} a_{x}^{j,\pm} \big( \tilde h_j^\pm(u_{x}) + (\tilde h_j^\pm)'(u_{x}) \bar\eta_x \big) \, dt \bigg] + \e_3,
\end{multline}
where the error term $\e_3$ is given by
\begin{equation*} 
  \e_3 := C K \int_0^{T_1}  H(\mu_t^N | \nu_t^N) \, dt + C N^{d-\epsilon_3}
\end{equation*}
for some $C, \epsilon_3 > 0$ independent of $N$.

Next we simplify the integrand in the right-hand side of \eqref{pf:4}. Since $c^\pm$ are independent of $\eta_0$, we obtain
\begin{align*}
  \tilde h_1^\pm(\b) = \beta E^{\nu_\beta}[c^\pm(\eta)], \qquad 
  \tilde h_2^\pm(\b) = E^{\nu_\beta}[c^\pm(\eta)].
\end{align*}
Recalling $a_{x}^{2, \pm} = - u_x a_{x}^{1, \pm}$, we observe that 
\[
  \sum_{j=1}^2 a_{x}^{j,\pm} \tilde h_j^\pm(u_{x}) = 0
  \quad \text{and} \quad
  \sum_{j=1}^2 a_{x}^{j,\pm} (\tilde h_j^\pm)'(u_{x}) = a_{x}^{1,\pm} E^{\nu_\beta}[c^\pm(\eta)] \big|_{\b=u_x}.
\]
Then, recalling \eqref{eq:f}, 
\begin{multline*}
\sum_{\pm} \sum_{j=1}^2 a_{x}^{j,\pm} (\tilde h_j^\pm)'(u_{x})
= K \left(\frac{ E^{\nu_\beta}[c^+(\eta)]}{u_x} -
\frac{E^{\nu_\beta}[c^-(\eta)]}{1 - u_x}\right)\Big|_{\b=u_x}  \\
 = \frac1{\chi(u_x)} \Big( E^{\nu_\beta}[c^+(\eta)] (1-u_x) - E^{\nu_\beta}[c^-(\eta)] u_x \Big) \Big|_{\b=u_x} 
= \frac{f(u_x)}{\chi(u_x)}.
\end{multline*}

In the computations that follow, we apply Theorem \ref{BG} in a similar manner. We will therefore abbreviate the explanation above by saying that Theorem \ref{BG} allows us to replace the contribution of the third term of the right-hand side of \eqref{eq:cor2.6} with respect to the expectation in \eqref{pf:3} by 
\begin{equation*}
  \sum_{x\in\T_N^d}  \E^{\mu^N} \bigg[ \int_0^{T_1} K f(u_x) \frac{\bar\eta_x}{\chi(u_x)} \, dt \bigg]
\end{equation*}
at the cost of the error term $\e_3$. In addition, the computation of the integrand is \eqref{pf:4} is mostly the reverse of the expansion made in \eqref{pf:5}. This will also be the case in the computations that follow; we will therefore omit this part of the computation.

\paragraph{First term of the right-hand side of \eqref{eq:cor2.6}.}
Since this term is quadratic in $\bar\eta_x$, it will not give a contribution to the leading order term in \eqref{pf:3}. We show this by applying Theorem \ref{BG} in a similar manner as above. For the first term of the right-hand side of \eqref{eq:cor2.6}, we take the sum over $y$ in front (by writing it as $\sum_{\pm e_i}$; a sum with $2d$ terms), and treat each term separately. We assume for convenience that $y = x + e_i$; the term corresponding to $y = x - e_i$ can be treated analogously. This term is given by
\begin{equation*}
  \sum_{x\in\T_N^d} -\frac{ N^2 (u_y-u_x)^2 }{ 2 \chi(u_x) \chi(u_y) } c_{x,y}(\eta) \bar \eta_x \bar \eta_y
   = \sum_{j=1}^4 a_{x}^j \tau_x h^j(\eta),
\end{equation*}
where
\[
  h^1(\eta) = c_{0, e_i}(\eta) \eta_0 \eta_{e_i}, \quad
  h^2(\eta) = c_{0, e_i}(\eta) \eta_{e_i}, \quad
  h^3(\eta) = c_{0, e_i}(\eta) \eta_0, \quad
  h^4(\eta) = c_{0, e_i}(\eta)
\]
and
\[
  a_x^1 = -\frac{ N^2 (u_y-u_x)^2 }{ 2 \chi(u_x) \chi(u_y) }, \quad
  a_x^2 = -u_x a_x^1, \quad
  a_x^3 = -u_y a_x^1, \quad
  a_x^4 = u_x u_y a_x^1.
\]
By  \eqref{uN:uf:bds} and \eqref{uN:der:bds:1} 
we have $|a_x^j| \leq CK^{2/\si}$, and thus Theorem \ref{BG} applies with $\theta = \frac2\sigma$. Since $c_{0,e_i}$ does not depend on $\eta_0$ or $\eta_{e_i}$ by assumption (4),
we obtain 
$$
\sum_{j=1}^4 a_{x}^j \tilde h^j(\b) = a_{x}^1 E^{\nu_\beta}[c_{0,e_i}(\eta)] (\b-u_x)(\b-u_y).
$$
This value vanishes at $\beta = u_x$. Moreover, using again \eqref{uN:der:bds:1}, we obtain  
$$
\bigg| \frac d{d\b} \sum_{j=1}^4 a_{x}^j \tilde h^j(\b) \bigg| \bigg|_{\beta = u_x} 
= \big| a_x^1 E^{\nu_\b} [c_{0,e_i}(\eta)] (u_x - u_y) \big| \big|_{\beta = u_x} 
\leq C K^{3/\si} / N.
$$
Hence, using Theorem \ref{BG} we bound the contribution of the first term of the right-hand side of \eqref{eq:cor2.6} to the expectation in \eqref{pf:3} in absolute value by 
\begin{equation} \label{err1}
  \e_1 := C \bigg( K^{3/\si} N^{d-1} + N^{d - \epsilon_1} + K^{2/\si} \int_0^{T_1} H(\mu_t^N|\nu_t^N) \, dt \bigg).
\end{equation}
for some $C, \epsilon_1 > 0$ independent of $N$.

\paragraph{Second term of the right-hand side of \eqref{eq:cor2.6}.}
In order to absorb the prefactor $N^2$ in the deterministic coefficients, we decompose this term as $ I_1+I_2+I_3$  
where 
\begin{align*}
I_1=& \frac{N^2}2 \sum_x \frac1{\chi(u_x)} \sum_{y:|x-y|=1} c_{x,y}(\eta) (\eta_x-\eta_y) (u_y-u_x), \\
I_2=& \frac{N^2}2 \sum_x \sum_{y:|x-y|=1} c_{x,y}(\eta) \eta_y \Big(\frac1{\chi(u_x)} - \frac1{\chi(u_y)}\Big)
 (u_y-u_x), \\
I_3=& \frac{N^2}2 \sum_x \sum_{y:|x-y|=1} c_{x,y}(\eta)\Big(\frac{u_y}{\chi(u_y)} - \frac{u_x}{\chi(u_x)}\Big)
 (u_y-u_x). 
\end{align*}
Here, we simply denote $\Si_x$ for $\Si_{x\in \T_N^d}$ and apply a similar convention for $y$.
Using the gradient condition (5), we rewrite $I_1$ as
\begin{align*}
I_1=  \frac{N^2}2 \sum_x  \frac1{\chi(u_x)} \sum_{i=1}^d 
\sum_{y=x\pm e_i} (\t_x h_i-\t_y h_i)(u_y-u_x).
\end{align*} 
Then, changing summation variables to put $\tau_x h_i$ in front, we obtain
\begin{align*}
I_1
&= \frac{N^2}2 \sum_{i=1}^d \sum_x \t_x h_i \Big( \frac1{\chi(u_x)} [u_{x+e_i}    + u_{x-e_i} - 2 u_x] \\
&\hspace{40mm}   + \frac1{\chi(u_{x-e_i})} [u_{x-e_i} - u_x] 
   + \frac1{\chi(u_{x+e_i})} [u_{x+e_i} - u_x] \Big) \\
&= \sum_{i=1}^d \sum_x \t_x h_i \bigg[ 
\frac{N^2}{\chi(u_x)} [u_{x+e_i}    + u_{x-e_i} - 2 u_x] + 
\sum_\pm \frac{N^2}2
 \Big( \frac1{\chi(u_{x \pm e_i})} - \frac1{\chi(u_x)} \Big) [u_{x \pm e_i} - u_x] \bigg] \\
 &= I_1^2 + I_1^+ + I_1^-.
\end{align*}
Then, for each $i$ separately, we apply Theorem \ref{BG} to each of the three terms. Recall from \eqref{eq:P} and assumption (6) that for each of these three terms
$$  
\tilde h_i (\beta) = E^{\nu_\beta} [h_i(\eta)] = P_i(\beta) = P(\beta).
$$
The coefficients corresponding to $I_1^2$ and $I_1^\pm$ are respectively 
\begin{align*}  
  a_x^2 &= \frac{N^2}{\chi(u_x)} [u_{x+e_i}    + u_{x-e_i} - 2 u_x], \\
  a_x^\pm &= \frac{N^2}2
 \Big( \frac1{\chi(u_{x \pm e_i})} - \frac1{\chi(u_x)} \Big) [u_{x \pm e_i} - u_x].
\end{align*} 
By \eqref{uN:der:bds:1},\eqref{uN:der:bds:2}
and the fact that the function $\frac1\chi : [u_-, u_+] \to \R$ is Lipschitz continuous, we have $|a_x^2|, |a_x^\pm| \leq CK^{2/\si}$.

Then, applying Theorem \ref{BG} to the contributions of $I_1^2$ and $I_1^\pm$ to the expectation and the integral in $t$ in \eqref{pf:3} and afterwards reverting the decomposition of $I_1$, we obtain that $I_1$ can be replaced by $J_{1,1} + J_{1,2}$, where
\begin{align*}
  J_{1,1} &=  \frac{N^2}2 \sum_x  \frac1{\chi(u_x)} \sum_{y:|x-y|=1} (P(u_x)-P(u_y))(u_y-u_x), \\
J_{1,2} &=  \frac{N^2}2 \sum_x  \frac1{\chi(u_x)} 
\sum_{y:|x-y|=1} \big( P'(u_x) \bar\eta_x -P'(u_y)\bar\eta_y \big) (u_y-u_x),
\end{align*} 
and the error made by this replacement equals 
\begin{equation} \label{err2}
  \e_2 := C N^{d - \epsilon_2} + C K^{2/\si} \int_0^{T_1} H(\mu_t^N|\nu_t^N) \, dt
\end{equation}
for some $C, \epsilon_2 > 0$. 

Similarly, we apply Theorem \ref{BG} to $I_2$ and $I_3$. The treatment for these terms is easier, because almost no rewriting is necessary to select appropriate choices for $a_x$ and $h$. The only rewriting we employ is for $I_2$, where the symmetry in $x$ and $y$ allows for replacing $\eta_y$ by $\eta_x$. Then, for $I_2$ we take $\tau_x h_2(\eta) := c_{x,y}(\eta) \eta_x$
with $y=x\pm e_i$,
and note from \eqref{Pp:pos} and the reversibility assumption (4) that
\[
  \tilde h_2 (\beta) = E^{\nu_\beta} [c_{0,\pm e_i}(\eta) \eta_0] = P'(\beta) \beta.
\]
For $I_3$ we simply take $\tau_x h_3(\eta) := c_{x,y}(\eta)$ with $y=x\pm e_i$, which results by a similar computation in
\[
  \tilde h_3 (\beta) =  P'(\beta).
\]
The coefficients corresponding to $I_2$ and $I_3$ are respectively (overwriting previous notation)
\begin{align*}  
  a_x^2 &= \frac{N^2}2 \Big(\frac1{\chi(u_x)} - \frac1{\chi(u_y)}\Big)
 (u_y-u_x), \\
  a_x^3 &= \frac{N^2}2 \Big(\frac{u_y}{\chi(u_y)} - \frac{u_x}{\chi(u_x)}\Big)
 (u_y-u_x),
\end{align*} 
where $y=x\pm e_i$.
As for $I_1$, it follows from the Lipschitz bound on the functions $\frac1{\chi(u)}$ and $\frac u{\chi(u)}$ that $|a_x^2|, |a_x^3| \leq CK^{2/\si}$.

Then, applying Theorem \ref{BG} to the contribution of $I_2 + I_3$ to the expectation in \eqref{pf:3}, we obtain that $I_2$ and $I_3$ can be replaced respectively by $J_{2,1} + J_{2,2}$ and $J_{3,1} + J_{3,2}$, where
\begin{align*}  
J_{2,1}=& - \frac{N^2}2 \sum_x \sum_{y:|x-y|=1} P'(u_y)u_y
  \Big(\frac1{\chi(u_y)} -\frac1{\chi(u_x)} \Big) (u_y-u_x), \\
J_{2,2}=& - \frac{N^2}2 \sum_x \sum_{y:|x-y|=1} \bar\eta_x \big(P'(u_x)+P''(u_x)u_x\big)    
  \Big(\frac1{\chi(u_y)} -\frac1{\chi(u_x)} \Big) (u_y-u_x), \\
J_{3,1}=& \frac{N^2}2 \sum_x \sum_{y:|x-y|=1} P'(u_x)
  \Big(\frac{u_y}{\chi(u_y)} -\frac{u_x}{\chi(u_x)} \Big) (u_y-u_x), \\
J_{3,2}=& \frac{N^2}2 \sum_x \sum_{y:|x-y|=1} \bar\eta_x P''(u_x)  
  \Big(\frac{u_y}{\chi(u_y)} -\frac{u_x}{\chi(u_x)} \Big) (u_y-u_x),
\end{align*}
and the error made by this replacement is the same as $\e_2$ in \eqref{err2} for possibly different constants $C, \epsilon_2 > 0$. 

Next we rewrite the six terms $J_{j,k}$. We claim that $J_{1,1}+J_{2,1}+J_{3,1} = O(K^{3/\sigma} N^{d-1})$, 
so that its contribution can be absorbed in $\e_1$ in \eqref{err1}.
To prove this claim, we combine the sums in these three terms, and observe that the coefficient of $\frac{N^2}2 (u_y-u_x) = O(K^{1/\sigma} N)$ in the summand equals
\begin{align*}
& \frac1{\chi(u_x)} (P(u_x)-P(u_y)) 
- P'(u_y)u_y  \Big(\frac1{\chi(u_y)} -\frac1{\chi(u_x)} \Big)
+  P'(u_x)  \Big(\frac{u_y}{\chi(u_y)} -\frac{u_x}{\chi(u_x)} \Big) \\
&= \frac1{\chi(u_x)}(P(u_x)-P(u_y))
- P'(u_y)u_y \frac{\chi(u_x) - \chi(u_y)}{\chi(u_x)\chi(u_y)} \\
& \hskip 30mm +  P'(u_x)\frac{\chi(u_y)(u_y - u_x) +(\chi(u_x) - \chi(u_y))u_y}{\chi(u_x)\chi(u_y)}  \\
&= \frac1{\chi(u_x)}\Big( P'(u_x)(u_x-u_y) + O(|\nabla_{xy}u|^2) \Big)
- P'(u_y)u_y \frac{\chi'(u_x)(u_x-u_y) + O(|\nabla_{xy}u|^2)}{\chi(u_x)\chi(u_y)} \\
& \hskip 30mm +  P'(u_x)\frac{\chi(u_y)(u_y - u_x) +\chi'(u_x)(u_x - u_y)u_y
 + O(|\nabla_{xy}u|^2)}{\chi(u_x)\chi(u_y)}  \\
& =  (P'(u_x)- P'(u_y)) u_y \frac{\chi'(u_x)(u_x - u_y)}{\chi(u_x)\chi(u_y)}  
+ O(|\nabla_{xy}u|^2)\\
& = O(|\nabla_{xy}u|^2) = O(K^{2/\sigma} N^{-2}).
\end{align*}

Next we treat the remaining three terms. $J_{1,2}$ can be rearranged as 
\begin{align*}
J_{1,2}= \frac{N^2}2 \sum_x \sum_{y:|x-y|=1} \bar\eta_x P'(u_x) 
  \Big(\frac1{\chi(u_y)} +\frac1{\chi(u_x)} \Big) (u_y-u_x).
\end{align*}
Summing this together with $J_{2,2}$ and $J_{3,2}$, the coefficient of $\frac{N^2}2 \bar\eta_x (u_y-u_x) = O(K^{1/\sigma} N)$ in the summand
becomes
\begin{align*}
 & P'(u_x) \Big(\frac1{\chi(u_y)} +\frac1{\chi(u_x)} \Big)
 - \big(P'(u_x)+P''(u_x)u_x\big) \Big(\frac1{\chi(u_y)} -\frac1{\chi(u_x)} \Big) 
 + P''(u_x) \Big(\frac{u_y}{\chi(u_y)} -\frac{u_x}{\chi(u_x)} \Big)  \\
 & =  2 \frac{P'(u_x)}{\chi(u_x)} + P''(u_x)\frac{u_y-u_x}{\chi(u_x)} + P''(u_x) (u_y-u_x) \Big(\frac1{\chi(u_y)} -\frac1{\chi(u_x)} \Big).
\end{align*}
The third term is $O(K^{2/\sigma} N^{-2})$. For the first two terms we recognize the first and second order term in the Taylor expansion of $P$ at $u_x$, so that these two terms can be replaced by 
\begin{equation*}
  \frac2{\chi(u_x)} \frac{P(u_y) - P(u_x)}{u_y - u_x} + O(|u_y - u_x|^2).
\end{equation*}
There is no danger for division by zero thanks to the prefactor $\frac{N^2}2 \bar\eta_x (u_y-u_x)$. Note also that $|u_y-u_x|^2 = O(K^{2/\sigma}N^{-2})$. In conclusion, 
\begin{align*}
  J_{1,2} + J_{2,2} + J_{3,2}
  &= \sum_x \bigg( N^2 \sum_{y:|x-y|=1} (P(u_y) - P(u_x)) \bigg) \frac{\bar \eta_x}{\chi(u_x)} + O(K^{3/\sigma} N^{d-1}) \\
  &= \sum_x \Delta^N P(u_x) \frac{\bar \eta_x}{\chi(u_x)}  + O(K^{3/\sigma} N^{d-1}).
\end{align*}

Finally, putting everything together, the contribution of the second term of the right-hand side of \eqref{eq:cor2.6} to the expectation in \eqref{pf:3} can be replaced by 
\begin{equation*}
  \sum_{x\in\T_N^d}  \E^{\mu^N} \bigg[ \int_0^{T_1} \Delta^N P(u_x) \frac{\bar \eta_x}{\chi(u_x)} \, dt \bigg]
\end{equation*}
at the cost of an error term of the form \eqref{err1}.

\paragraph{Conclusion.}
By substituting our computations above for the contribution of each of the four terms of the right-hand side of \eqref{eq:cor2.6} in the expectation in \eqref{pf:3}, we obtain that for all $N$ large enough
\begin{multline*} 
  H(\mu_{T_1}^N|\nu_{T_1}^N)
  \le \E^{\mu^N} \bigg[ \int_0^{T_1} \sum_{x\in\T_N^d} \big[ \Delta^N P(u_x(t)) + K f(u_x(t)) - \partial_t u_x(t) \big] \frac{\bar \eta_x(t)}{\chi(u_x(t))} \, dt \bigg] \\
  + C \bigg( K^{3/\si} N^{d-1} + N^{d - \epsilon} + K^{2/\si} \int_0^{T_1} H(\mu_t^N|\nu_t^N) \, dt \bigg)
\end{multline*}
for some constants $C, \epsilon > 0$ independent of $N$. Since $u_x(t)$ satisfies equation \eqref{eq:HD-discre}, the integrand in the first term vanishes. Then, since $T_1 \in (0, T]$ is arbitrary, Theorem \ref{thm:EstHent} follows from Gronwall's lemma after taking $\delta > 0$ in $K \leq \delta (\log N)^{\si/2}$ small enough with respect to $T$, $C$ and $\epsilon$.

\section{Proof of the Boltzmann-Gibbs principle}
\label{s:BG:pf}

We prove Theorem \ref{BG} in Section \ref{s:BG:pf:pf}. Its main two building blocks are Lemmas \ref{l:L10:6} and \ref{l:L10:9}, which we state and prove in Section \ref{s:BG:pf:lemmas}. The proof of Lemma \ref{l:L10:6} is loosely based on \cite[{Section}\ 7]{EFHPS} and the proof of Lemma \ref{l:L10:9} relies on an equivalence of ensembles estimate and on a concentration inequality (cf.\ \cite[Lemma 3.6]{FT}). Since the time dependence of $f_{t,x}$ and $u_{t,x}$ is not essential in any of the computations that follow, we drop the $t$ from the subscript. The time dependence of the coefficients $a_{t,x}$, however, will have a significant impact on the proof (as can be anticipated from \eqref{BG:assn:atx}), and therefore we keep it in the subscript of $a$. 

\subsection{Key lemmas}
\label{s:BG:pf:lemmas}

We split
\begin{equation} \label{BG:pf:0}
  \E^{\mu^N} \bigg |\int_0^T \sum_{x \in \T_N^d} a_{t,x} f_x dt \bigg |
  \leq \E^{\mu^N} \bigg |\int_0^T \sum_{x \in \T_N^d} a_{t,x} m_x dt \bigg |
  + \E^{\mu^N} \bigg |\int_0^T \sum_{x \in \T_N^d} a_{t,x} E^{\nu_\beta}[f_x \mid \eta_x^\ell] dt \bigg |,
\end{equation}
where
\begin{align*}
\beta &:= \frac12, \\
  m_x &:= f_x - E^{\nu_\beta}[f_x \mid \eta_x^\ell], \\
  \eta_x^\ell &:= \frac1{\ell_*^d} \sum_{y \in \Lambda_{\ell,x}} \eta_y,
  \\ 
  \ell_* &:= 2 \ell + 1, \\
  \Lambda_{\ell,x} &:= \{ y \in \T_N^d : |y - x| \leq \ell \}
\end{align*}
and $\ell = \ell(N)$ is such that $1 \ll \ell \ll N$ as $N \to \infty$. Note that $|\Lambda_{\ell,x}| = \ell_*^d$ and that $\eta_x^\ell$ is a local average of $\eta$. We estimate both terms in \eqref{BG:pf:0} separately in Lemmas \ref{l:L10:6} and \ref{l:L10:9}.

\begin{lem} \label{l:L10:6}
Let $\gamma = \gamma(N) > 0$ and $\e = \e(N) \in (0, T/2)$. If  
\begin{equation} \label{l:L10:6:assn}
   \gamma \e^{-\Theta} K^\theta \ell_*^{d+2} \ll N^2
  \quad \text{as } N \to \infty, 
\end{equation} 
then for all $N$ large enough
  \[
      \E^{\mu^N} \bigg |\int_0^T \sum_{x \in \T_N^d} a_{t,x} m_x dt \bigg |
      \leq C N^d \Big( \e^{1 - \Theta} K^\theta + \frac K\gamma + \frac{ \gamma  \ell_*^{d+2} K^{2\theta} }{ \e^{2\Theta} N^2 } \Big),
  \]
where the constant $C$ is independent of $N$.
\end{lem} 

\begin{proof}
Applying the entropy inequality with respect to $\nu_\beta$ and with the given constant $\gamma$ we obtain
\begin{equation} \label{BG:pf:5}
  \E^{\mu^N} \bigg |\int_0^T \sum_{x \in \T_N^d} a_{t,x} m_x dt \bigg |
  \leq \frac1\gamma H(\mu^N | \nu_\beta) + \frac1\gamma \log \E^{\nu_\beta} \bigg[ \exp \bigg| \gamma \int_0^T \sum_{x \in \T_N^d} a_{t,x} m_x  dt \bigg| \bigg].
\end{equation}
For the first term, we estimate (inserting $\beta = \frac12$)
\[
  H(\mu^N | \nu_\beta)
  = H(\mu^N | \nu_\beta)
  \leq \max_{\eta \in \mathcal X_N} \log \frac{d \mu^N}{d \nu_\beta}(\eta)
  \leq \max_{\eta \in \mathcal X_N} \log \dfrac{1}{\nu_\beta(\eta)}
  = N^d \log 2,
\]
where we used $\mu^N(\eta)\le1$ and $\nu_\beta(\eta)=2^{-N^d}$ for any $\eta\in\mathcal X_N$.
To bound the second term in \eqref{BG:pf:5}, we split the integral over $(0,T)$ into three integrals over the intervals $(0, \e)$, $(\e, T-\e)$ and $(T-\e, T)$. For the integrals over $(0, \e)$ and $(T-\e, T)$, we simply use \eqref{BG:assn:atx} and the fact that $m_x$ is a uniformly bounded random variable to estimate the summand as
\begin{equation} \label{BG:pf:6}
  \sup_{x \in \T_N^d} \sup_{\eta \in \mathcal X_N} |a_{t,x} m_x(\eta)|
  \leq C \frac{ K^\theta }{[t(T-t)]^\Theta}.
\end{equation}
Then, we bound the integral over $(0, \e)$ as
\begin{equation*}
  \bigg| \int_0^\e \sum_{x \in \T_N^d} a_{t,x} m_x dt \bigg|
  \leq C K^\theta N^d \int_0^\e \frac1{t^\Theta} \, dt
  = C_\Theta K^\theta N^d \e^{1-\Theta}.
\end{equation*}
The integral over $(T-\e, T)$ satisfies a similar bound. Hence, the second term in \eqref{BG:pf:5} is bounded from above by
\begin{equation} \label{BG:pf:7}
  \frac1\gamma \log \E^{\nu_\beta} \bigg[ \exp \bigg| \gamma \int_\e^{T-\e} \sum_{x \in \T_N^d} a_{t,x} m_x  dt \bigg| \bigg]
  + C K^\theta N^d \e^{1-\Theta}.
\end{equation}

Next we bound the first term in \eqref{BG:pf:7}. We first use $e^{|x|} \leq e^x + e^{-x}$, and then apply the Feynman-Kac formula (see \cite[Appendix 1, Lemma 7.2]{KL}, whose proof does not require $\nu_\beta$ to be an invariant measure of $L_N$). This yields  
\begin{multline*}
  \log \E^{\nu_\beta} \bigg[ \exp \bigg| \gamma \int_\e^{T-\e} \sum_{x \in \T_N^d} a_{t,x} m_x  dt \bigg| \bigg] \\
  \leq \sup_\pm \int_\e^{T-\e} \sup_g \Big( \gamma \sum_{x \in \T_N^d} \langle \pm a_{t,x} m_x, g \rangle_{\nu_\beta} - \langle -L_N \sqrt g, \sqrt g \rangle_{\nu_\beta} \Big) dt + \log 2
\end{multline*}
where the supremum over $g$ is over all densities with respect to $\nu_\beta$. The constant $\log 2$ can be absorbed in the second error term in Lemma \ref{l:L10:6}; we neglect it in the remainder. We also neglect $\sup_\pm$, because the proof below works verbatim when $m_x$ is replaced by $-m_x$. Then, reflecting on \eqref{BG:pf:5}, we obtain
\begin{multline} \label{pf:1}
  \E^{\mu^N} \bigg |\int_0^T \sum_{x \in \T_N^d} a_{t,x} m_x dt \bigg| \\
  \leq \frac{N^d}\gamma {\log 2}
  + C K^\theta N^d \e^{1-\Theta}
  + \frac1\gamma \int_\e^{T-\e} \sup_g \Big( \gamma \sum_{x \in \T_N^d} \langle a_{t,x} m_x, g \rangle_{\nu_\beta} - \langle -L_N \sqrt g, \sqrt g \rangle_{\nu_\beta} \Big) dt.
\end{multline} 

It is left to bound the integral in the right-hand side of \eqref{pf:1}. With this aim, we set $\psi := \sqrt g$ and expand
\begin{equation} \label{BG:pf:2}
  \langle -L_N \psi, \psi \rangle_{\nu_\beta}
  = N^2 \langle -L_K \psi, \psi \rangle_{\nu_\beta} + K \langle -L_G \psi, \psi \rangle_{\nu_\beta}.
\end{equation}
We start with bounding the Glauber part from below. By the definition of $L_G$, 
\begin{align*}
  \langle -L_G \psi, \psi \rangle_{\nu_\beta}
  &= \sum_{x \in \T_N^d} E^{\nu_\beta} {\big[} c_x(\eta) \big\{ \psi(\eta)^2 - \psi(\eta)\psi(\eta^x) \big\} {\big]} \\
  &\geq \frac12 \sum_{x \in \T_N^d} E^{\nu_\beta} {\big[} c_x(\eta) \big\{ \psi(\eta)^2 - \psi(\eta^x)^2 \big\} {\big]} \\
  &\geq - \frac12 \Big( \max_{\eta \in \mathcal X_N} c(\eta) \Big) \sum_{x \in \T_N^d} E^{\nu_\beta} {\big[} \psi(\eta^x)^2 {\big]}.
\end{align*}
Since the flip rate $c$ is a local function, the prefactor above is finite and independent of $N$. For the summand,
recalling that $\beta = \frac12$ and that $\psi^2 = g$ is a density with respect to $\nu_\beta$, we obtain
\[
  E^{\nu_\beta} {\big[} \psi(\eta^x)^2 {\big]}
  = E^{\nu_\beta} {\big[} \psi(\eta)^2 {\big]}
  = 1.
\]
Hence,
\begin{align*}
  \langle -L_G \psi, \psi \rangle_{\nu_\beta}
  \geq - C \sum_{x \in \T_N^d} 1
  = - C N^d.
\end{align*}

Next we treat the Kawasaki term in \eqref{BG:pf:2}. Recalling the uniform lower bound on $c_{x,y}(\eta)$ in \eqref{cxy:LB},  
\begin{align*}
  \langle -L_K \psi, \psi \rangle_{\nu_\beta}
  &= \frac14 \sum_{x\in\T_N^d} \sum_{|y - x| = 1}  \int_{\mathcal{X}_N} c_{x,y}(\eta) \{\psi(\eta^{x,y}) - \psi(\eta)\}^2 d\nu_\beta  \\
  &\geq \frac1{4C} \sum_{x\in\T_N^d} \sum_{|y - x| = 1} \underbrace{ \int_{\mathcal{X}_N} \{\psi(\eta^{x,y}) - \psi(\eta)\}^2 d\nu_\beta. }_{E_{x,y} (\psi)}
\end{align*}
The right-hand side is the Dirichlet form of the SSEP (symmetric simple exclusion process). While bounding it from below is standard, we need a quantitative estimate, and therefore we provide this estimate in detail.

Next, we localize the Dirichlet form of the SSEP:
\begin{align*}
  \frac1{4C} \sum_{x\in\T_N^d} \sum_{|y - x| = 1} E_{x,y} (\psi)
  &= \frac1{4C} \sum_{x\in\T_N^d} \sum_{|y - x| = 1} \frac1{\ell_*^d} \sum_{z \in \Lambda_{\ell,x}} E_{x,y} (\psi) \\
  &= \frac1{C \ell_*^d} \sum_{z\in\T_N^d} \frac14 \sum_{x \in \Lambda_{\ell,z}} \sum_{|y - x| = 1} E_{x,y} (\psi) \\
  &\geq \frac1{C \ell_*^d} \sum_{z\in\T_N^d} {\bigg(}\frac14 \sum_{\substack{ x,y \in \Lambda_{\ell,z} \\ |y - x| = 1 }} E_{x,y} (\psi) {\bigg)}
  =: \frac1{C \ell_*^d} \sum_{z\in\T_N^d} D_{\ell, z}(\psi),
\end{align*}
where the inequality is justified by removing several bonds at the boundary of $\Lambda_{\ell,x}$ and using that $E_{x,y} {(\psi)} \geq 0$. Note that 
\[
  D_{\ell, z}(\psi)
  = \frac14 \sum_{\substack{ x,y \in \Lambda_{\ell,z} \\ |y - x| = 1 }} \int_{\mathcal{X}_N} \{\psi(\eta^{x,y}) - \psi(\eta)\}^2 d\nu_\beta
  = E^{\nu_\beta} [\psi(- L_{\ell, z} \psi)],
\]
where $L_{\ell, z}$ is the generator of the SSEP restricted to $\Lambda_{\ell,z}$. 
Using the estimate above on $\langle -L_K \psi, \psi \rangle_{\nu_\beta}$, we obtain from \eqref{BG:pf:2} 
\[
  \langle -L_N \psi, \psi \rangle_{\nu_\beta}
  \geq \frac{N^2}{C \ell_*^d} \sum_{x\in\T_N^d} D_{\ell, x}(\psi) - C'KN^d.
\]
Using this, the supremum in \eqref{pf:1} is bounded by
\begin{multline} \label{BG:pf:3}  
  \sup_g \bigg( \gamma \sum_{x \in \T_N^d} \langle a_{t,x} m_x, g \rangle_{\nu_\beta} - \langle -L_N \sqrt g, \sqrt g \rangle_{\nu_\beta} \bigg) \\
  \leq \sum_{x \in \T_N^d} \sup_g \Big( \gamma \langle a_{t,x} m_x, g \rangle_{\nu_\beta} - \frac{N^2}{C \ell_*^d} D_{\ell, x}(\sqrt g) \Big)
    + C'KN^d.
\end{multline}

Next, we further bound from above the supremum in the right-hand side. The SSEP generated by $L_{\ell, x}$ conserves the number of particles, and is therefore reducible on $\mathcal \{0,1\}^{\Lambda_{\ell,x}}$. To make it irreducible, we condition on the number of particles. For this purpose, for each $\ell\in\N$, $x\in\T_N^d$ and $j\le\ell_*^d$, let
\begin{align*}
   {\mathcal X_{\ell, x,j}}
  &:= \Big\{ \eta \in \{0, 1\}^{\Lambda_{\ell,x}} : \sum_{{y} \in \Lambda_{\ell,x}} {\eta_y} = j \Big\}, && \\
  L_{\ell, x, j} f(\eta) 
  &:= L_{\ell, x} f(\eta) = \frac12 \sum_{ \substack{ y, z\in \Lambda_{\ell,x}  \\ |y-z|=1 }}
\left\{  f\left(  \eta^{y,z}\right)  -f\left(\eta\right)  \right\} 
  && \text{for all } f:\mathcal X_{\ell, x,j}\to\R, \\
  \nu_{\ell, x, j} 
  &:= \nu_\beta \Big( \cdot \mid \sum_{{y} \in \Lambda_{\ell,x}} {\eta_y} = j \Big)  \quad \text{and} && \\
  D_{\ell, x, j}(f) 
  &:= E^{\nu_{\ell, x, j}} [f(- L_{\ell, x} f)] 
  && \text{for all } f:\mathcal X_{\ell, x,j}\to\R
\end{align*}
be respectively the corresponding configuration space, generator, canonical invariant measure and Dirichlet form. 
Then, conditioning in the supremum in the right-hand side of \eqref{BG:pf:3} the measure $g \nu_\beta$ on the event $\{\sum_{y\in\Lambda_{\ell,x}} \eta_y = j\}$, and then taking $\max_j$, we estimate this supremum from above by
\begin{equation} \label{BG:pf:4} 
  \frac{N^2}{C \ell_*^d} \max_{j \leq \ell_*^d} \sup_g \Big( C \frac{\gamma  \ell_*^d}{N^2} \langle a_{t,x} m_x, g \rangle_{\nu_{\ell, x, j}} - D_{\ell, x, j}(\sqrt g)  \Big),
\end{equation}
where now $g$ is a density with respect to $\nu_{\ell, x, j}$. Since $|a_{t,x} m_x|$ is uniformly bounded by $C K^\theta \e^{-\Theta}$ for $t \in [\e, T-\e]$ (recall \eqref{BG:pf:6}), the conditions for the Rayleigh estimate in \cite[Appendix 3, Theorem 1.1]{KL} are met. Applying this estimate, we obtain the following upper bound on \eqref{BG:pf:4} 
\begin{equation} \label{BG:pf:1}
  C \frac{\gamma^2 \ell_*^d}{N^2} \max_{j \leq \ell_*^d} \frac{ \langle (-L_{\ell, x, j})^{-1} (a_{t,x} m_x), a_{t,x} m_x \rangle_{\nu_{\ell, x, j}} }{ 1 - C' K^\theta \e^{-\Theta} \gamma \ell_*^d / (N^2 \gap(j,\ell)) },
\end{equation}
provided that the denominator is strictly positive. Here, $\gap(j,\ell)$ is the spectral gap of $L_{\ell,x,j}$. By \cite[Corollary A.1]{FUY} (or, alternatively, \cite[Section 8]{Qua}) we have $\gap(j,\ell) \geq C / \ell_*^2$. Then, from the given bound in \eqref{l:L10:6:assn} it follows that the denominator in \eqref{BG:pf:1} is larger than $\frac12$ for all $N$ large enough. For the numerator, we bound
\[
  \langle (-L_{\ell, x, j})^{-1} (a_{t,x} m_x), a_{t,x} m_x \rangle_{\nu_{\ell, x, j}}
  \leq \frac1{\gap(j,\ell)} \sup_{t,x,\eta} | a_{t,x} m_x |^2
  \leq C K^{2\theta} \e^{-2\Theta} \ell_*^2.
\]
In conclusion, for all $N$ large enough, \eqref{BG:pf:1} is bounded from above by
\begin{align*}
  C \frac{\gamma^2 \ell_*^{d+2} K^{2\theta}}{\e^{2\Theta} N^2}. 
\end{align*}
Collecting all estimates above, Lemma \ref{l:L10:6} follows.
\end{proof}

Next, in Lemma \ref{l:L10:9} below, we bound the second term in \eqref{BG:pf:0}. In preparation for this, we establish an equivalence of ensembles lemma. To state it in our context, recall $\beta = \frac{1}{2}$, the function $f_x$ (cf.\ \eqref{f_tx eq}) and is expectation (cf.\ \eqref{ht})
$$
  \tilde f_x (\rho) = E^{\nu_\rho}[f_x]
  \qquad \text{for all } 0 \leq \rho \leq 1.
$$

\begin{lem}[Equivalence of ensembles] \label{l:GJ13:P31}
There exists $C > 0$ such that for all $N$ large enough
  \[
      \max_{x \in \T_N^d} \sup_{\eta \in \mathcal X_N} \left| E^{\nu_\beta}[f_x \mid \eta_x^\ell]
      - \tilde f_x (\eta_x^\ell) + \frac{\chi(\eta_x^\ell)}{2 \ell_*^d} \frac{\partial^2 \tilde f_x}{\partial \rho^2 } (\eta_x^\ell) \right| 
      \leq \frac C{\ell^{2d}}.
  \]
\end{lem}

\begin{proof}
In a one-dimensional setting, Lemma \ref{l:GJ13:P31} is stated and proved in \cite[Proposition 3.1]{GJ}. In what follows, we mention the minor modifications by which this proof extends to the higher dimensional setting in Lemma \ref{l:GJ13:P31}. 

Since $x$ plays no important role, we choose it as $0$, and write $f = f_x$, $\eta^\ell = \eta_x^\ell$ and $\Lambda_\ell = \Lambda_{\ell, x}$. Since $f$ has finite range, we can express it as a finite sum
\[
  f(\eta)
  = \sum_{A \subset \supp f} f_A \prod_{x \in A} \eta_x,
\]
where $f_A \in \R$ are constants. We take $N$ large enough so that $\supp f \subset \Lambda_\ell$. Since the inside of the absolute values in the estimate in Lemma \ref{l:GJ13:P31} is linear as an operator on $f$, it is enough to prove the lemma only for the functions $\eta \mapsto \prod_{x \in A} \eta_x$. Moreover, since $\nu_\beta$ and $\eta^\ell$ are invariant under swapping sites inside $\supp f$, it is enough to prove Lemma \ref{l:GJ13:P31} only for the functions
\[
  g_k(\eta) := \prod_{j=1}^k \eta_{x_j}
  \qquad k = 1,\ldots,L
\]
instead of $f$, where $\{x_1, \ldots, x_L\} = \supp f$ and $L = |\supp f|$.

Since $g_k$ is a finite list of functions, and since $\nu_\beta$ and $\eta^\ell$ are independent from the geometry of the bonds in $\Lambda_\ell$, the remainder of the proof can be copied from the proof of \cite[Proposition 3.1]{GJ}. We omit the details. 
\end{proof}

\begin{lem} \label{l:L10:9}
For all $N$
\begin{multline*}
      \E^{\mu^N} \bigg |\int_0^T \sum_{x \in \T_N^d} a_{t,x} E^{\nu_\beta}[f_x \mid \eta_x^\ell] dt \bigg | \\
      \leq C K^\theta \int_0^T  H(\mu_t^N | \nu_t^N) \frac{dt}{ [t(T-t)]^\Theta } 
  + C' K^\theta N^d \Big( \frac1{\ell_*^d} + \frac{K^{2/\sigma} \ell^2}{N^{1-\Theta}} \Big),
\end{multline*}
where the constants $C, C' > 0$ are independent of $N$. 
\end{lem}

\begin{proof}
Note from the given bound on the coefficients $a_{t,x}$ in \eqref{BG:assn:atx} that
\begin{equation*} 
\E^{\mu^N} \bigg |\int_0^T \sum_{x \in \T_N^d} a_{t,x} E^{\nu_\beta}[f_x \mid \eta_x^\ell] dt \bigg |
\leq C K^\theta \int_0^T E^{\mu_t^N} {\Big[}  \sum_{x \in \T_N^d} \big| E^{\nu_\beta}[f_x \mid \eta_x^\ell] \big| {\Big]} \frac{dt}{ [t(T-t)]^\Theta }.
\end{equation*}
By Lemma \ref{l:GJ13:P31}, this is bounded from above by
\begin{equation}\label{4-1}
  C K^\theta \int_0^T  E^{\mu_t^N} {\bigg[}  \sum_{x \in \T_N^d} \Big| \tilde f_x (\eta_x^\ell) - \frac{\chi(\eta_x^\ell)}{2 \ell_*^d} \frac{\partial^2 \tilde f_x}{\partial \rho^2 } (\eta_x^\ell) \Big| {\bigg]} \frac{dt}{ [t(T-t)]^\Theta } + C' \frac{ K^\theta N^d }{ \ell^{2d} } \int_0^T \frac{dt}{ [t(T-t)]^\Theta }.
\end{equation}
The second integral is an $N$-independent constant which scales as $T^{1 - 2\Theta}$. For the first integral, we use the explicit form of $f_x$ to obtain for all $0 \leq \rho \leq 1$ that
\begin{align*}
  \tilde f_x(\rho) 
  &= \tilde h(\rho) -\tilde h(u_x) -\tilde h'(u_x)(\rho - u_x) 
  = O\left((\rho-u_x)^2 \right), \\
  \frac{\partial^2 \tilde f_x}{\partial \rho^2 }(\rho) &= \frac{\partial^2 \tilde h}{\partial \rho^2 } (\rho) = O(1).
\end{align*}
Then, setting
\begin{align*}
\bar \eta_x^\ell := \eta_x^\ell - u_x^\ell, \quad  u_x^\ell = \dfrac{1}{\ell_*^d}\sum_{y \in \Lambda_{\ell,x}}u_y,
\end{align*}
we estimate the summand by
\[
  \bigg | \tilde f_x (\eta_x^\ell) - \frac{\chi(\eta_x^\ell)}{2 \ell_*^d} \frac{\partial^2 \tilde f_x}{\partial \rho^2 } (\eta_x^\ell) \bigg |
  \leq C \big( \bar \eta_x^\ell \big)^2 + C (u_x^\ell - u_x)^2 + C' / \ell_*^d.
\]
Using this we bound the first term in \eqref{4-1} from above by
\begin{equation}\label{4-2}
  C K^\theta \int_0^T  E^{\mu_t^N} {\Big[} \sum_{x \in \T_N^d} \big( \bar \eta_x^\ell \big)^2 {\Big]} \frac{dt}{ [t(T-t)]^\Theta } 
  + C K^\theta \int_0^T \sum_{x \in \T_N^d} (u_x^\ell - u_x)^2 \frac{dt}{ [t(T-t)]^\Theta }
  + C' K^\theta \frac{N^d}{\ell_*^d}.
\end{equation}

To bound the second second term in \eqref{4-2}, we combine the uniform bound on $u_x$ with the bound on the gradient in \eqref{eq:nablau} to obtain
\begin{align*}
  |u_x^\ell - u_x|
  \leq \frac1{\ell_*^d} \sum_{y \in \Lambda_{\ell,x}} |u_y - u_x|
  &\leq \frac C{\ell_*^d} \sum_{y \in \Lambda_{\ell,x}} d \ell \max_{1 \leq i \leq d} |u_{x + e_i} - u_x| \\
  &\leq C \ell \min \Big\{ \frac{K^{1/\sigma}}{N [t(T-t)]^\Theta}, 1 \Big\}.
\end{align*}
Then, setting $t_N := (K^{1/\sigma} / N)^{1/\Theta}$ and assuming that $N$ is large enough, we estimate
\[
  C K^\theta \int_0^T \sum_{x \in \T_N^d} (u_x^\ell - u_x)^2 \frac{dt}{ [t(T-t)]^\Theta }
  \leq C' K^\theta \ell^2 N^d \bigg( \int_0^{t_N} t^{-\Theta} dt + \frac{K^{2/\sigma}}{N^2} \int_{t_N}^{T/2} t^{-3\Theta} dt \bigg).
\]
We estimate the integrals in the right-hand side as
\[
  \int_0^{t_N} t^{-\Theta} dt 
  = C t_N^{1 - \Theta} 
  \leq C' \frac{K^{1/\sigma}}{N^{1-\Theta}}
\]
and
\[
  \int_{t_N}^{T/2} t^{-3\Theta} dt
  \leq \int_{N^{-1/\Theta}}^{T/2} t^{-3\Theta} dt
  \leq C N^{1+\Theta}.
\]
Putting these estimates together, we conclude that
\[
  C K^\theta \int_0^T \sum_{x \in \T_N^d} (u_x^\ell - u_x)^2 \frac{dt}{ [t(T-t)]^\Theta }
  \leq C' K^{\theta + 2/\sigma} \ell^2 N^d / N^{1-\Theta}.
\]

For the first term in \eqref{4-2}, we use the entropy inequality with constant $\lambda > 0$ to obtain
\[
  E^{\mu_t^N} {\Big[} \sum_{x \in \T_N^d} \big( \bar \eta_x^\ell \big)^2 {\Big]}
  \leq \frac1\lambda H(\mu_t^N | \nu_t^N) 
  +  \frac1\lambda \log E^{\nu_t^N} {\bigg[} \exp \bigg( \lambda \sum_{x \in \T_N^d} \big( \bar \eta_x^\ell \big)^2 \bigg) {\bigg]}.
\]
For the second term, we apply the usual argument for taking $\sum_x$ in front, which relies on H\"older's inequality and the observation that $\bar \eta_x^\ell$ and $\bar \eta_y^\ell$ are independent if $| x - y |_\infty \geq 2\ell_*$. Details can be found in \cite[Lemma 4.2 and Lemma B.4]{JM1}. This yields 
\[
  \frac1\lambda \log E^{\nu_t^N} {\bigg[} \exp \bigg( \lambda \sum_{x \in \T_N^d} \big( \bar \eta_x^\ell \big)^2 \bigg) {\bigg]}
  \leq \frac1{2^d \lambda \ell_*^d} \sum_{x \in \T_N^d} \log E^{\nu_t^N} {\Big[} \exp \Big( 2^d \lambda \ell_*^d \big( \bar \eta_x^\ell \big)^2 \Big) {\Big]}.
\]

Next we apply the concentration inequality \cite[Lemma 3.6]{FT}. With this aim, we set $\lambda = 2^{-d}$. Since the random variables $\{ \eta_y - u_y \}_{y \in \Lambda_{\ell,x}}$ are independent, have mean zero under $\nu_t^N$, and have values in the interval $[-u_x, 1 - u_x]$ of length $1$, the concentration inequality states that 
\[
  \log E^{\nu_t^N} {\Big[} \exp \Big( 2^d \lambda \ell_*^d \big( \bar \eta_x^\ell \big)^2 \Big) {\Big]}
  = \log E^{\nu_t^N} {\bigg[} \exp \bigg( \ell_*^{-d} \bigg( \sum_{y \in \Lambda_{\ell,x}} (\eta_y - u_y) \bigg)^2 \bigg) {\bigg]}
  \leq 2.
\]

In conclusion, the first term of \eqref{4-2} is bounded from above by 
\[
  C K^\theta \int_0^T  H(\mu_t^N | \nu_t^N) \frac{dt}{ [t(T-t)]^\Theta } 
  + C' K^\theta \frac{N^d}{\ell_*^d}.
\]
This completes the proof of Lemma \ref{l:L10:9}.
\end{proof}

\subsection{Proof of Theorem \ref{BG}} 
\label{s:BG:pf:pf}

Next we show that Theorem \ref{BG} follows by applying the estimates in Lemmas \ref{l:L10:6} and \ref{l:L10:9} to \eqref{BG:pf:0}. Indeed, the first term in the right-hand side of \eqref{BG:est} comes directly from the first term in the right-hand side of the estimate in Lemma \ref{l:L10:9}. It is left to show that the other error terms in Lemmas \ref{l:L10:6} and \ref{l:L10:9} can be bounded by $C N^{d - \epsilon_1}$ for some $\epsilon_1 > 0$ for an appropriate choice of $\ell, \e, \gamma$. We demonstrate this with some simple algebra by choosing $\ell = N^{\alpha_\ell}$, $\frac1\e = N^{\alpha_\e}$, $\gamma = N^{\alpha_\gamma}$ and assuming that $K \leq C N^{\alpha_K}$ for certain constants $\alpha_\ell, \alpha_\e, \alpha_\gamma, \alpha_K > 0$ which we choose below. The restrictions $1 \ll \ell \ll N$ and \eqref{l:L10:6:assn} are met if
\begin{align} \label{BG:pf:8}
  \alpha_\ell < 1
  \quad \text{and} \quad
  \alpha_\gamma + \Theta \alpha_\e + \theta \alpha_K + (d+2) \alpha_\ell < 2.
\end{align}
For the five remaining error terms (three in Lemma \ref{l:L10:6}, and the latter two in Lemma \ref{l:L10:9}) to be bounded by $C N^{d - \epsilon_1}$, the conditions on the exponents are respectively
\begin{subequations} \label{BG:pf:9} 
\begin{align} \label{BG:pf:9a}
  \theta \alpha_K &< (1 - \Theta) \alpha_\e, \\\label{BG:pf:9b}
  \alpha_K &< \alpha_\gamma, \\\label{BG:pf:9c}
  \alpha_\gamma + (d+2)\alpha_\ell + 2 \theta \alpha_K + 2 \Theta \alpha_\e &< 2, \\\label{BG:pf:9d}
  \theta \alpha_K &< d \alpha_\ell, \\\label{BG:pf:9e}
  \Big( \theta + \frac2\sigma \Big) \alpha_K + 2 \alpha_\ell &< 1-\Theta.  
\end{align}
\end{subequations}
Note that \eqref{BG:pf:8}, \eqref{BG:pf:9c} and \eqref{BG:pf:9e} are met if
\begin{equation}\label{BG:pf:10}
 \alpha_\gamma + \Big( d + \frac4{1-\Theta} \Big) \alpha_\ell + \frac2{1 - \Theta} \Big( \theta + \frac2\sigma \Big) \alpha_K + 2 \Theta \alpha_\e < 2.
\end{equation}
To satisfy conditions \eqref{BG:pf:9a}, \eqref{BG:pf:9b} and \eqref{BG:pf:9d}, we take
\[
   \alpha_\e = 2 \frac\theta{1 - \Theta} \alpha_K,
   \quad \alpha_\gamma = 2 \alpha_K
   \quad \text{and} \quad \alpha_\ell = \frac{ 2 \theta }d \alpha_K.
\]
Substituting these choices in \eqref{BG:pf:10} and using that $\Theta < 1$, $\sigma < 1$ and $d \geq 1$, we observe that \eqref{BG:pf:10} is satisfied if
\[
  \alpha_K = \frac{1 - \Theta}{10 \theta + 3\sigma^{-1}}.
\]
Hence, all the error terms in Lemmas \ref{l:L10:6} and \ref{l:L10:9} (other than the term involving the entropy) can be bounded by $C N^{d - \epsilon_1}$ for some $\epsilon_1 > 0$.

\end{document}